\documentclass[11pt]{amsart}
\usepackage[utf8]{inputenc}
\usepackage{todonotes}
\usepackage[margin=1in]{geometry}
\usepackage{enumitem}

\title{5-cycles in the Complement of Minimal Prime Graphs}
\author{Micah Dorton, Thomas Michael Keller, Ryan Tang, Justin Yu}

\usepackage{amsmath,amssymb,amsthm,graphicx,mathtools,comment,color,colonequals}

\usepackage{tikz}
\usetikzlibrary{calc}

\usepackage{graphicx}

\usepackage{xcolor}

\usepackage{float} 

\renewcommand{\comment}[1]{} 

\numberwithin{equation}{section}

\newtheorem{thm}[equation]{Theorem}

\newtheorem{lemma}[equation]{Lemma}

\theoremstyle{definition}
\newtheorem{remark}[equation]{Remark}

\newtheorem{definition}[equation]{Definition}

\newtheorem{prop-def}{Proposition-Definition}

\begin{document} 
\begin{abstract}
Minimal prime graphs (MPGs) are a special class of prime graphs (also known as  Gruenberg-Kegel graphs) associated with finite solvable groups. A graph is an MPG if it has at least two vertices, is connected, its complement is triangle-free and 3-colorable, and the addition of an edge to the complement will violate triangle-freeness or 3-colorability. In this paper, we continue the study of the complements of MPGs focusing on their cycle structure. Our main result establishes that every edge in the complement of an MPG is contained in a 5-cycle. This finding is a much stronger form of an older
result stating that every minimal prime graph complement contains at least one induced 5-cycle.
\end{abstract}
\maketitle
\section{Introduction}
The Gruenberg-Kegel graph, or prime graph, of a finite 
group $G$ is the graph whose vertices are the prime 
divisors of $|G|$, with an edge between distinct primes $p$ and $q$ if and only if
$G$ contains an element of order $pq$. Since their introduction in the 1970s, 
prime graphs have continuously been an object of interest.\\

A 2015 result by Gruber, Keller, Lewis, Naughton, and Strasser \cite{gruber} 
characterizes the prime graphs of solvable groups as follows: a simple graph is isomorphic to the 
prime graph of a finite solvable group if and only if its complement is triangle-free and 
3-colorable. This characterization has inspired a line of research into special 
subclasses of prime graphs, particularly those with extremal or minimal properties.

The focus of this paper are the \emph{minimal prime graphs} (MPGs) as introduced in 
\cite{gruber} and studied in more detail in \cite{florez}. These are connected graphs on at least two vertices whose 
complements are triangle-free and 3-colorable, and for which the addition of any edge 
to the complement violates one of these properties. On the group theoretical side, the prime graph
of a solvable group is an MPG if it has a maximal number of Frobenius 
actions among its Sylow subgroups.\\

It was already noted in \cite{gruber} that MPGs 
(and thus their complements) always have an induced 5-cycle. 
Later work \cite{huang} introduced 
new methods for generating MPGs beyond vertex duplication, 
and in \cite{kou} it is shown that the minimal degree of the
complement of an MPG is at least 2, and MPGs with a vertex 
of degree 2 are completely classified in \cite{micah}.\\

In the main result of this paper, we establish a far-reaching
and as we believe, 
unexpected strengthening of the above statement of the 
existence of 5-cycles in MPGs as follows:\\

{\bf Theorem:}
\textit{Every edge in the 
complement of a minimal prime graph is contained in a 5-cycle}. \\

Groups whose prime graph complement (and thus their
prime graph as well) is a 5-cycle have a unique Frobenius
digraph (in the sense of \cite{gruber}) and thus a Sylow structure as described in \cite[Theorem 2.2(a)]{keller1994}. (Note that by the Hartley-Turull lemma \cite[Theorem 3.31]{isaacs} any MPG 
can be realized by a group with abelian Fitting subgroup
of square-free order.) The main
result of this paper shows that this type of subgroup pervades
groups whose prime graph is an MPG.\\

The proof of our main result involves constructing and analyzing a family of 
subgraphs $\Gamma(m,n,k,l,x,y)$, and showing that any edge 
either belongs to a 5-cycle or leads to a larger subgraph 
of the same type, enabling an inductive argument.\\

Our main result puts MPGCs in the family of triangle-free 
graphs where every edge belongs to a 5-cycle. This family
has drawn some interest, see e.g. \cite{mathoverflow}.
It seems
impossible to classify this larger family of graphs,
but it contains some
prominent examples such as the Petersen graph which
is indeed an MPGC, and also the generalized Petersen
graphs $P(n,2)$. In general, though, this class of graphs
seems huge. However, with the extra restrictions that come
with MPGCs, there remains hope that one day it might be possible to 
completely classify them.\\

\textbf{Acknowledgments}: The authors are grateful for the support of this research by the Mathworks program at Texas State University. They also thank Ronok Ghosal for some discussions on the project.

\section{Notation and Definitions}
\begin{itemize}[leftmargin=0.5cm]
    \item The vertex and edge set of $\Gamma$ are denoted $V_{\Gamma}, E_{\Gamma},$ respectively.
    \item $\overline\Gamma$ denotes the complement of $\Gamma$
    \item A cycle of length $n$ in a graph $\Gamma$ is denoted $(c_1, c_2 \ldots, c_n)$ where $c_i\in V_\Gamma$ and $(c_i, c_{i + 1}) \in E_\Gamma$ for $i = 1, 2, \ldots n$ where $c_{n + 1} = c_1.$ 
    \item A path of length $n$ in a graph $\Gamma$ is denoted $\{c_1, c_2 \ldots, c_n\}$ where $c_i\in V_\Gamma$ and $(c_i, c_{i + 1}) \in E_\Gamma$ for $i = 1, 2, \ldots n - 1$. 
    \item For indices, unless otherwise specified: 
    \begin{itemize}[leftmargin=0.5cm]
        \item Characters with no subscript, e.g. $i, j, k$, are fixed.
        \item Characters with a subscript, e.g. $n_1, n_2$, are meant to be considered for all values within the range.
    \end{itemize}
\end{itemize}
\begin{definition}(Minimal Prime Graphs)
    A minimal prime graph (MPG) $\Gamma$ is a graph with 2 or more vertices satisfying:
    \begin{enumerate}[leftmargin=0.5cm]
        \item $\Gamma$ is connected
        \item $\overline\Gamma$ is triangle-free
        \item $\overline\Gamma$ is 3-colorable
        \item Removing any edge from $\Gamma$ will violate (2) or (3). This is often referred to as the maximality condition because of how we work in the complement. 
    \end{enumerate}
\end{definition}
\begin{definition}
    Let $\Gamma$ be an MPG, and $u, v, h \in V_\Gamma$ such that $u,v$ are different colors. The statement ``the vertex $h$ blocks the edge $(u,v)$'' means that $(u, v)\not\in E_{\overline\Gamma},$ but the edges $(h, u), (h, v) \in E_{\overline\Gamma}.$ (This is from the maximality condition on triangle-freeness)
\end{definition}
From here, we refer to the complement of an MPG as an MPGC.

\section{Preliminaries}
\begin{thm}\label{connected}
    MPGCs are connected.
\end{thm}
\begin{proof}
    Assume for the sake of contradiction that there exists some MPG $\Gamma$ where $\overline\Gamma$ is disconnected and can be split into subgraphs $\Pi_1$ and $\Pi_2$. Pick any vertex $v_1 \in V_{\Pi_1},$ and $v_2 \in V_{\Pi_2}.$ Since $\overline\Gamma$ is disconnected, there does not exist an edge between $v_1$ and $v_2$. By the definition of MPG, we have that the maximality condition implies that, no edge can be added to $\overline\Gamma$ without violating either triangle-freeness or 3-colorability.

    Now, consider adding the edge $(v_1, v_2)$ into $\overline\Gamma.$ 
    
    There cannot be a violation of triangle-freeness, or else $\overline\Gamma$ is not disconnected, which is a contradiction. 
    
    Now, we prove there cannot be a violation of 3-colorability. For every coloring of $\overline\Gamma$ not violating 3-colorability, note that we can cycle the colors of $\Pi_1$ without violating 3-colorability, which means we can change red to blue, blue to green, and green to red. We start by picking any coloring of $\overline\Gamma.$ 
    \begin{itemize}[leftmargin=0.5cm]
        \item If the colors of $v_1, v_2$ are different, then the edge $(v_1,v_2)$ can be added without violating 3-colorability or triangle-freeness, leading to a contradiction to the maximality condition as we assumed that $\Gamma$ is an MPG.
        \item If the colors of $v_1, v_2$ are the same, we can cycle the colors once in $\Pi_1,$ which will result in $v_1$ and $v_2$ being colored differently, implying that the edge $(v_1, v_2)$ can be added, which leads to a contradiction.
    \end{itemize}

    Thus, in all cases there is a contradiction which means that no MPG exists where its complement is disconnected.
\end{proof}
The next two results are \cite[Corollary 11]{kou} and \cite[Proposition 12]{kou}, respectively.
\begin{thm}\label{deg2}
    In an MPGC, each vertex has degree of at least 2.
\end{thm}
\begin{thm}\label{dia23} 
    MPGCs have diameter 2 or 3. 
\end{thm}
\section{Main result}
We begin by proving the following result which will be beneficial in the future. 
\begin{lemma} \label{7-cycle}
    Let there be a 7-cycle that is a subgraph of an MPGC. Pick a valid coloring of the MPGC. If each of the 3 colors occur at least twice in the 7-cycle, then each edge in this 7-cycle subgraph is contained in a 5-cycle. 
\end{lemma}
\begin{proof}
    There are only two distinct colorings (up to an isomorphism) such that each of the 3 colors show up more than once:
    
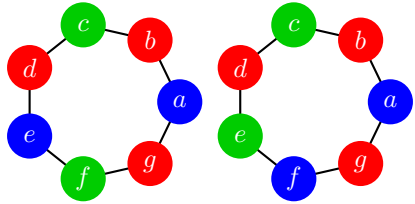
\begin{center}
\begin{tikzpicture}[
    scale=0.7, 
    dot/.style={
        circle, 
        minimum size=6mm, 
        inner sep=0pt,    
        font=\small\itshape, 
        text=white
    },
    line/.style={thick, black}
]

    \definecolor{green!80!black}{RGB}{0, 100, 0}

    \begin{scope}[xshift=-2cm]
        \foreach \i [count=\n from 0] in {a,b,c,d,e,f,g} {
            \coordinate (L\i) at ({360/7 * \n}:1.5);
        }
        \draw[line] (La) -- (Lb) -- (Lc) -- (Ld) -- (Le) -- (Lf) -- (Lg) -- cycle;
        
        \node[dot, fill=blue] at (La) {$a$};
        \node[dot, fill=red] at (Lb) {$b$};
        \node[dot, fill=green!80!black] at (Lc) {$c$};
        \node[dot, fill=red] at (Ld) {$d$};
        \node[dot, fill=blue] at (Le) {$e$};
        \node[dot, fill=green!80!black] at (Lf) {$f$};
        \node[dot, fill=red] at (Lg) {$g$};
    \end{scope}

    \begin{scope}[xshift=2cm]
        \foreach \i [count=\n from 0] in {a,b,c,d,e,f,g} {
            \coordinate (R\i) at ({360/7 * \n}:1.5);
        }
        \draw[line] (Ra) -- (Rb) -- (Rc) -- (Rd) -- (Re) -- (Rf) -- (Rg) -- cycle;
        
        \node[dot, fill=blue] at (Ra) {$a$};
        \node[dot, fill=red] at (Rb) {$b$};
        \node[dot, fill=green!80!black] at (Rc) {$c$};
        \node[dot, fill=red] at (Rd) {$d$};
        \node[dot, fill=green!80!black] at (Re) {$e$}; 
        \node[dot, fill=blue] at (Rf) {$f$};       
        \node[dot, fill=red] at (Rg) {$g$};
    \end{scope}

\end{tikzpicture}
    \end{center} 
    Note that the only difference is that the color of $e$ and $f$ are swapped. In the argument that follows, we will consider both cases at the same time, except for a few cases. Notice that in the example graphs, the colorings of vertices $e$ and $f$ have 2 colors, to demonstrate these two cases being considered at once.

    We will do casework on the edge $(a, d):$
    \begin{enumerate}[leftmargin=0.5cm]
        \item If $(a,d)$ exists, then the 5-cycle $(a, d, e, f, g)$ contains all edges but $(a, b)$, $(b, c)$, $(c, d)$. Now look at the edge $(c, g):$
        
        If $(c, g)$ does not exist, then suppose the blue vertex $h$ blocks the edge. If $h=e,$ then $(c,d,e)$ is a triangle. If $h = a,$ then $(a,b,c)$ is a triangle. The 5-cycles $(a, b, c, h, g)$, $(g, a, d, c, h)$ cover the edges. 
        
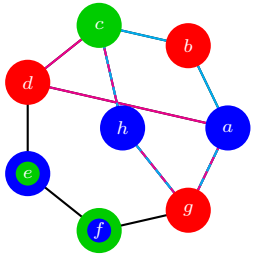
\begin{center}
\begin{tikzpicture}[
    scale=0.7,
    dot/.style={
        circle, 
        minimum size=6mm, 
        inner sep=0pt, 
        font=\tiny, 
        text=white,
        line width=0.4pt 
    },
ring_dot/.style={
        circle, 
        minimum size=6mm - 4pt, 
        inner sep=0pt, 
        font=\tiny, 
        text=white,
        fill=green!80!black,   
        draw=blue,         
        line width=4pt   
    },
ring_dot_2/.style={
        circle, 
        minimum size=6mm - 4pt, 
        inner sep=0pt, 
        font=\tiny, 
        text=white,
        fill=blue,   
        draw=green!80!black,         
        line width=4pt   
    }
]

    \definecolor{green!80!black}{RGB}{0, 100, 0}

    \coordinate (h) at (0,0);
    \foreach \n [count=\i from 0] in {a,b,c,d,e,f,g} {
        \coordinate (\n) at ({360/7 * \i}:2);
    }

    \draw[thick] (a) -- (b) -- (c) -- (d) -- (e) -- (f) -- (g) -- cycle;
    \draw[thick] (a) -- (d);
    \draw[thick] (c) -- (h) -- (g);

    \draw[cyan, thick] (a) -- (b) -- (c) -- (h) -- (g) -- cycle;
    \draw[magenta, thick] (c) -- (d) -- (a) -- (g) -- (h) -- cycle;
    \draw[cyan, dashed, thick] (c) -- (h) -- (g) -- (a);

    \node[dot, fill=blue] at (a) {$a$};
    \node[dot, fill=red] at (b) {$b$};
    \node[dot, fill=green!80!black] at (c) {$c$};
    \node[dot, fill=red] at (d) {$d$};
    \node[ring_dot, fill=green!80!black] at (e) {$e$};
    \node[ring_dot_2, fill=blue] at (f) {$f$};
    \node[dot, fill=red] at (g) {$g$};
    \node[dot, fill=blue] at (h) {$h$};

\end{tikzpicture}
        \end{center}

        Now consider if $(c, g)$ does exist. We do casework on the edge $(b,e).$ 

        If $(b,e)$ does not exist, let $h$ block the edge. Then, $h$ is either green or blue(dependent on $e$). If $h = c,$ $(c,d,e)$ is a triangle. If $h = a,$ $(a,d,e)$ is a triangle. If $h = f$ or $h$ is a new vertex, then $(b,c,d,e,h)$, $(a,b,h,e,d)$ covers the rest of the edges. 

        If $(b,e)$ exists, then $(c,d,e,f,g)$, $(b,c,g,f,e)$, $(b, e, f, g, a)$ finish. 
        \item If $(a, d)$ does not exist, there exists a green vertex $h$ that will block the edge $(a, d)$. If $h = c$, then $(a,b,c)$ is a triangle. If $h = f$, then $(d,e,f)$ is a triangle. 
        
        We will do casework on $h = e$ or $h$ is a new vertex. 
        \begin{itemize}[leftmargin=0.5cm]
            \item If $h = e,$ then notice $(e, d, c, b, a).$ We now do casework on the edge $(b,f)$:
            \begin{itemize}[leftmargin=0.5cm]
                \item If $(b, f)$ does not exist, let the green vertex $i$ block it. If $i = e,$ then $(a,b,e)$ is a triangle. If $i$ is a new vertex or if $i = c,$ observe the 5-cycles $(e, f, i, b, a)$ and $(f, g, a, b, i),$ which cover the rest of the edges. 
                
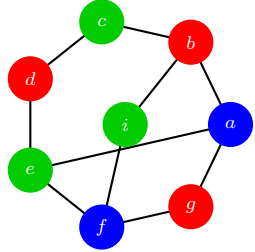
\begin{center}
\begin{tikzpicture}[
    scale=0.7,
    dot/.style={
        circle, 
        minimum size=6mm, 
        inner sep=0pt,    
        font=\tiny, 
        text=white,
        thick
    }
]

    \definecolor{green!80!black}{RGB}{0, 100, 0}

    \coordinate (j) at (0,0);
    \foreach \n [count=\i from 0] in {a,b,c,d,e,f,g} {
        \coordinate (\n) at ({360/7 * \i}:2);
    }

    \draw[thick] (a) -- (b) -- (c) -- (d) -- (e) -- (f) -- (g) -- cycle;
    
    \draw[thick] (e) -- (a);
    \draw[thick] (b) -- (j) -- (f);

    \node[dot, fill=blue] at (a) {$a$};
    \node[dot, fill=red] at (b) {$b$};
    \node[dot, fill=green!80!black] at (c) {$c$};
    \node[dot, fill=red] at (d) {$d$};
    \node[dot, fill=green!80!black] at (e) {$e$};
    \node[dot, fill=blue] at (f) {$f$};
    \node[dot, fill=red] at (g) {$g$};
    
    \node[dot, fill=green!80!black] at (j) {$i$};

\end{tikzpicture}
                \end{center} 
                \item We have the 5-cycle $(b,c,d,e,f),$ which means we only need to prove $(f, g), (a, g)$ are covered by 5-cycles. 
            
                If $(c, g)$ exists, notice the 5-cycles $(c, d, e, a, g)$ and $(c, d, e, f, g)$ which cover the rest of the edges. 
                
                Thus, let $j$ be a blue vertex that blocks the edge $(c, g).$ If $j = f,$ $(b,c,f)$ is a triangle. If $j = a,$ $(a,b,c)$ is a triangle. Hence, assume $j$ is a new vertex. However, notice the 5-cycles $(c, j, g, f, b)$ and $(a, b, c, j, g)$ which cover the rest of the edges. 
                
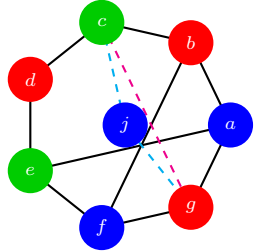
\begin{center}
\begin{tikzpicture}[
    scale=0.7,
    dot/.style={
        circle, 
        minimum size=6mm, 
        inner sep=0pt,    
        font=\tiny, 
        text=white,
        thick
    }
]

    \definecolor{green!80!black}{RGB}{0, 100, 0}

    \coordinate (j) at (0,0);
    \foreach \n [count=\i from 0] in {a,b,c,d,e,f,g} {
        \coordinate (\n) at ({360/7 * \i}:2);
    }

    \draw[thick] (a) -- (b) -- (c) -- (d) -- (e) -- (f) -- (g) -- cycle;
    \draw[thick] (e) -- (a);
    \draw[thick] (b) -- (f);

    \draw[cyan, dashed, thick] (c) -- (j) -- (g);
    \draw[magenta, dashed, thick] (c) -- (g);

    \node[dot, fill=blue] at (a) {$a$};
    \node[dot, fill=red] at (b) {$b$};
    \node[dot, fill=green!80!black] at (c) {$c$};
    \node[dot, fill=red] at (d) {$d$};
    \node[dot, fill=green!80!black] at (e) {$e$};
    \node[dot, fill=blue] at (f) {$f$};
    \node[dot, fill=red] at (g) {$g$};
    \node[dot, fill=blue] at (j) {$j$};

\end{tikzpicture}
                \end{center} 
            \end{itemize}
            \item Thus, we examine the case where $h$ is a new vertex.
        
            First notice the 5-cycles $(a, b, c, d, h).$ Now we look at proving that the edges $(d, e)$, $(e, f)$,$(f, g)$, $(g, a)$ are contained in a 5-cycle. Consider the edge $(e, b).$ 
            
            If $(e, b)$ exists, then the 5-cycles $(b, a, g, f, e),$ and $(d, e, b, a, h)$ cover the rest.  
            
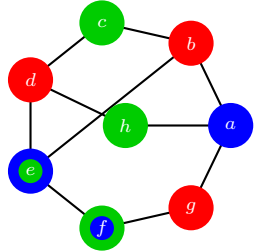
\begin{center}
\begin{tikzpicture}[
    scale=0.7,
    dot/.style={
        circle, 
        minimum size=6mm, 
        inner sep=0pt,    
        font=\tiny, 
        text=white,
        thick
    },
ring_dot/.style={
        circle, 
        minimum size=6mm - 4pt, 
        inner sep=0pt, 
        font=\tiny, 
        text=white,
        fill=green!80!black,   
        draw=blue,         
        line width=4pt   
    },
ring_dot_2/.style={
        circle, 
        minimum size=6mm - 4pt, 
        inner sep=0pt, 
        font=\tiny, 
        text=white,
        fill=blue,   
        draw=green!80!black,         
        line width=4pt   
    }]

    \definecolor{green!80!black}{RGB}{0, 100, 0}

    \coordinate (h) at (0,0);
    \foreach \n [count=\i from 0] in {a,b,c,d,e,f,g} {
        \coordinate (\n) at ({360/7 * \i}:2);
    }

    \draw[thick] (a) -- (b) -- (c) -- (d) -- (e) -- (f) -- (g) -- cycle;
    
    \draw[thick] (a) -- (h) -- (d);
    \draw[thick] (b) -- (e);

    \node[dot, fill=blue] at (a) {$a$};
    \node[dot, fill=red] at (b) {$b$};
    \node[dot, fill=green!80!black] at (c) {$c$};
    \node[dot, fill=red] at (d) {$d$};
    \node[ring_dot, fill=green!80!black] at (e) {$e$};
    \node[ring_dot_2, fill=blue] at (f) {$f$};
    \node[dot, fill=red] at (g) {$g$};
    \node[dot, fill=green!80!black] at (h) {$h$};

\end{tikzpicture}
            \end{center}
    
            Thus, we consider when $(e, b)$ does not exist. Since $(e, b)$ is symmetric to $(c, g),$ we only need to consider when $(c, g)$ does not exist. 
            
            Let $i$ be the vertex that blocks $(e, b)$ and $j$ be the blue vertex that blocks $(c, g)$ be $j$. We first handle the cases when $i, j$ are already part of the graph: 
            \begin{itemize}[leftmargin=0.5cm]
                \item If $j=e,$ we have $(c,d,e)$ and if $j=a,$ we have $(a,b,c).$ If $j = f,$ then consider $(e, a).$ If $(e, a)$ exists, we have $(e, a, b, c, d)$, $(e, a, b, c, f)$, and $(a,b,c,f,g)$. If $(e, a)$ does not exist, let the red vertex that blocks it be $k.$ If $k = d$, then $(a,h,d)$ is a triangle and if $k=g$, then $(e,f,g)$ is a triangle. If $k = b$ or if $k$ is a new vertex, then we have the 5-cycles $(e, k, a, g, f)$, $(e, k, a, h, d),$ which are unrelated to $b.$ 
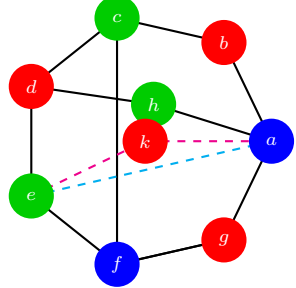
                \begin{center}
\begin{tikzpicture}[
    scale=0.7,
    dot/.style={
        circle, 
        minimum size=6mm, 
        inner sep=0pt,    
        font=\tiny, 
        text=white,
        thick
    }
]

    \definecolor{green!80!black}{RGB}{0, 100, 0}

    \coordinate (k) at (0,0);
    \foreach \n [count=\i from 0] in {a,b,c,d,e,f,g} {
        \coordinate (\n) at ({360/7 * \i}:2.4);
    }
    
    \coordinate (h) at ({360/7 * 1.5}:0.72);

    \draw[thick] (a) -- (b) -- (c) -- (d) -- (e) -- (f) -- (g) -- cycle;
    \draw[thick] (a) -- (h) -- (d);
    \draw[thick] (c) -- (f) -- (g);

    \draw[cyan, dashed, thick] (e) -- (a);
    \draw[magenta, dashed, thick] (e) -- (k) -- (a);

    \node[dot, fill=blue] at (a) {$a$};
    \node[dot, fill=red] at (b) {$b$};
    \node[dot, fill=green!80!black] at (c) {$c$};
    \node[dot, fill=red] at (d) {$d$};
    \node[dot, fill=green!80!black] at (e) {$e$};
    \node[dot, fill=blue] at (f) {$f$};
    \node[dot, fill=red] at (g) {$g$};
    \node[dot, fill=green!80!black] at (h) {$h$};
    
    \node[dot, fill=red] at (k) {$k$};

\end{tikzpicture}
                \end{center} 
                Thus, we only need to consider when $j$ is a new vertex. 
            \item If $i = f,$ then we have the 5-cycles $(c, d, e, f, b)$, $(a,b,c,j,g)$, $ (f, g, j, c, b)$ which cover each edge. 
                
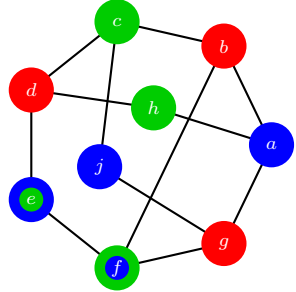
\begin{center}
\begin{tikzpicture}[
    scale=0.7,
    dot/.style={
        circle, 
        minimum size=6mm, 
        inner sep=0pt,    
        font=\tiny, 
        text=white,
        thick
    },
ring_dot/.style={
        circle, 
        minimum size=6mm - 4pt, 
        inner sep=0pt, 
        font=\tiny, 
        text=white,
        fill=green!80!black,   
        draw=blue,         
        line width=4pt   
    },
ring_dot_2/.style={
        circle, 
        minimum size=6mm - 4pt, 
        inner sep=0pt, 
        font=\tiny, 
        text=white,
        fill=blue,   
        draw=green!80!black,         
        line width=4pt   
    }
]

    \definecolor{green!80!black}{RGB}{0, 100, 0}

    \foreach \n [count=\i from 0] in {a,b,c,d,e,f,g} {
        \coordinate (\n) at ({360/7 * \i}:2.4);
    }
    
    \coordinate (h) at ({360/7 * 1.5}:0.72);
    
    \coordinate (j) at ({360/7 * 4}:0.96);

    \draw[thick] (a) -- (b) -- (c) -- (d) -- (e) -- (f) -- (g) -- cycle;
    
    \draw[thick] (a) -- (h) -- (d);
    \draw[thick] (c) -- (j) -- (g);
    \draw[thick] (f) -- (b);

    \node[dot, fill=blue] at (a) {$a$};
    \node[dot, fill=red] at (b) {$b$};
    \node[dot, fill=green!80!black] at (c) {$c$};
    \node[dot, fill=red] at (d) {$d$};
    \node[ring_dot, fill=green!80!black] at (e) {$e$};
    \node[ring_dot_2, fill=blue] at (f) {$f$};
    \node[dot, fill=red] at (g) {$g$};
    \node[dot, fill=green!80!black] at (h) {$h$};
    \node[dot, fill=blue] at (j) {$j$};

\end{tikzpicture}
                \end{center} 
                If $i = c,$ then $(c,d,e)$ is a triangle. If $i = h,$ then $(a,b,h)$ is a triangle. If $i = j,$ $(c,b,j)$ is a triangle. Finally, we consider if $i = a.$  First notice $(a, b, c, d, e)$, $(a, b, c, j, g)$ which cover all edges but $(e, f), (f, g).$ Now, we do casework on $(b,f).$
                
                If $(b, f)$ exists, we have the 5-cycles: $(b, f, e, d, c), (b, c, j, g, f).$ 
    
                If $(b, f)$ does not exist, let the vertex that blocks it be $k.$ If $k = e,$ $(a,b,e)$ is a triangle. If $k=c$, $h$, or a new vertex, then two 5-cycles $(a, b, k, f, g)$, $(a, b, k, f, e)$ finish. 
                
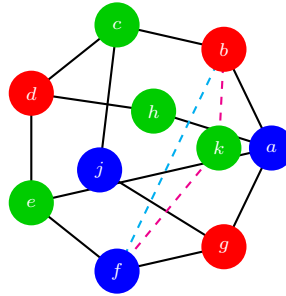
\begin{center}
\begin{tikzpicture}[
    scale=0.7,
    dot/.style={
        circle, 
        minimum size=6mm, 
        inner sep=0pt, 
        font=\tiny, 
        text=white,
        thick
    }
]

    \definecolor{green!80!black}{RGB}{0, 100, 0}

    \foreach \n [count=\i from 0] in {a,b,c,d,e,f,g} {
        \coordinate (\n) at ({360/7 * \i}:2.4);
    }
    
    \coordinate (h) at ({360/7 * 1.5}:0.72);
    
    \coordinate (j) at ({360/7 * 4}:0.96);
    
    \coordinate (k) at (1.4, 0);

    \draw[thick] (a) -- (b) -- (c) -- (d) -- (e) -- (f) -- (g) -- cycle;
    \draw[thick] (a) -- (h) -- (d);
    \draw[thick] (c) -- (j) -- (g);
    \draw[thick] (a) -- (e);

    \draw[cyan, dashed, thick] (b) -- (f);
    \draw[magenta, dashed, thick] (b) -- (k) -- (f);

    \node[dot, fill=blue] at (a) {$a$};
    \node[dot, fill=red] at (b) {$b$};
    \node[dot, fill=green!80!black] at (c) {$c$};
    \node[dot, fill=red] at (d) {$d$};
    \node[dot, fill=green!80!black] at (e) {$e$};
    \node[dot, fill=blue] at (f) {$f$};
    \node[dot, fill=red] at (g) {$g$};
    \node[dot, fill=green!80!black] at (h) {$h$};
    \node[dot, fill=blue] at (j) {$j$};
    \node[dot, fill=green!80!black] at (k) {$k$};

\end{tikzpicture}
                \end{center} 
            \end{itemize}
            Now, we can consider the cases when $i, j$ are new vertices.
    
            First observe the 5-cycles $(c, d, e, i, b), (a, b, c, j, g).$ Hence, we only need to prove that the edges $(e, f), (f, g)$ are contained in a 5-cycle. Now, consider the edge $(f, b).$
    
            If $(f, b)$ exists, then notice the 5-cycles $(f, g, j, c, b),(f, e, d, c, b).$ 
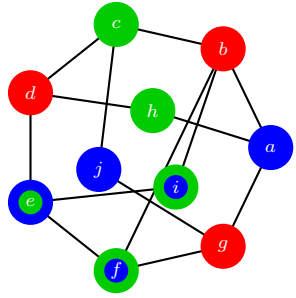
            \begin{center} 
\begin{tikzpicture}[
    scale=0.7,
    dot/.style={
        circle, 
        minimum size=6mm, 
        inner sep=0pt, 
        font=\tiny, 
        text=white,
        thick
    },
ring_dot/.style={
        circle, 
        minimum size=6mm - 4pt, 
        inner sep=0pt, 
        font=\tiny, 
        text=white,
        fill=green!80!black,   
        draw=blue,         
        line width=4pt   
    },
ring_dot_2/.style={
        circle, 
        minimum size=6mm - 4pt, 
        inner sep=0pt, 
        font=\tiny, 
        text=white,
        fill=blue,   
        draw=green!80!black,         
        line width=4pt   
    }]

    \definecolor{green!80!black}{RGB}{0, 100, 0}

    \foreach \n [count=\i from 0] in {a,b,c,d,e,f,g} {
        \coordinate (\n) at ({360/7 * \i}:2.4);
    }
    
    \coordinate (h) at ({360/7 * 1.5}:0.72);
    
    \coordinate (i) at ({360/7 * 6}:0.96);
    
    \coordinate (j) at ({360/7 * 4}:0.96);

    \draw[thick] (a) -- (b) -- (c) -- (d) -- (e) -- (f) -- (g) -- cycle;
    
    \draw[thick] (a) -- (h) -- (d);
    \draw[thick] (b) -- (i) -- (e);
    \draw[thick] (c) -- (j) -- (g);
    \draw[thick] (f) -- (b);

    \node[dot, fill=blue] at (a) {$a$};
    \node[dot, fill=red] at (b) {$b$};
    \node[dot, fill=green!80!black] at (c) {$c$};
    \node[dot, fill=red] at (d) {$d$};
    \node[ring_dot, fill=green!80!black] at (e) {$e$};
    \node[ring_dot_2, fill=blue] at (f) {$f$};
    \node[dot, fill=red] at (g) {$g$};
    
    \node[dot, fill=green!80!black] at (h) {$h$};
    \node[ring_dot_2, fill=blue] at (i) {$i$};
    \node[dot, fill=blue] at (j) {$j$};

\end{tikzpicture}
            \end{center} 
    
            If $(f, b)$ does not exist, then let the vertex that blocks the edge be $k.$ Note that $k$ is blue or green. If $k = e,$ $(a,b,e,f,g)$ finishes. If $k = a,$ $(a,f,g)$ exists. If $k = i$, then $(e,f,i)$ is a triangle. Hence, notice the 5-cycles $(f, g, a, b, k), (e, f, k, b, i)$ which finishes this case. 
            
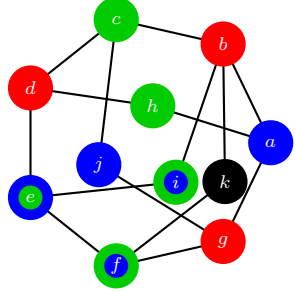
\begin{center}
\begin{tikzpicture}[
    scale=0.7,
    dot/.style={
        circle, 
        minimum size=6mm, 
        inner sep=0pt, 
        font=\tiny, 
        text=white,
        thick
    },
ring_dot/.style={
        circle, 
        minimum size=6mm - 4pt, 
        inner sep=0pt, 
        font=\tiny, 
        text=white,
        fill=green!80!black,   
        draw=blue,         
        line width=4pt   
    },
ring_dot_2/.style={
        circle, 
        minimum size=6mm - 4pt, 
        inner sep=0pt, 
        font=\tiny, 
        text=white,
        fill=blue,   
        draw=green!80!black,         
        line width=4pt   
    }]

    \definecolor{green!80!black}{RGB}{0, 100, 0}

    \foreach \n [count=\i from 0] in {a,b,c,d,e,f,g} {
        \coordinate (\n) at ({360/7 * \i}:2.4);
    }
    
    \coordinate (h) at ({360/7 * 1.5}:0.72); 
    \coordinate (i) at ({360/7 * 6}:0.96);   
    \coordinate (j) at ({360/7 * 4}:0.96);   
    
    \coordinate (k) at ({1080/7 + 180}:1.7);

    \draw[thick] (a) -- (b) -- (c) -- (d) -- (e) -- (f) -- (g) -- cycle;
    
    \draw[thick] (a) -- (h) -- (d);
    \draw[thick] (b) -- (i) -- (e);
    \draw[thick] (c) -- (j) -- (g);
    
    \draw[thick] (f) -- (k) -- (b);

    \node[dot, fill=blue] at (a) {$a$};
    \node[dot, fill=red] at (b) {$b$};
    \node[dot, fill=green!80!black] at (c) {$c$};
    \node[dot, fill=red] at (d) {$d$};
    \node[ring_dot, fill=green!80!black] at (e) {$e$};
    \node[ring_dot_2, fill=blue] at (f) {$f$};
    \node[dot, fill=red] at (g) {$g$};
    
    \node[dot, fill=green!80!black] at (h) {$h$};
    \node[ring_dot_2, fill=blue] at (i) {$i$};
    \node[dot, fill=blue] at (j) {$j$};
    \node[dot, fill=black] at (k) {$k$};

\end{tikzpicture}
            \end{center}
        \end{itemize}
    \end{enumerate}
    Thus, each edge in the 7-cycle is part of a 5-cycle. 
\end{proof}  
From here, all Lemmas/theorems which state that each edge is contained in a 5-cycle follow the same structure: 
\begin{enumerate}[leftmargin=0.5cm]
    \item We show that the original subgraph can not be an MPGC alone.
    \item We build onto the subgraph. We do this by adding vertices or edges.
    \item We finish our result. 
\end{enumerate}
Recall our notation at the start: For indices, unless otherwise specified: 
    \begin{itemize}[leftmargin=0.5cm]
        \item Characters with no subscript, e.g. $i, j, k$, are fixed.
        \item Characters with a subscript, e.g. $n_1, n_2$, are meant to be considered for all values within the range.
    \end{itemize}

We will now define a new type of graph:
\begin{definition}
    Let $m, n, x, y \in \mathbb N, k, l \in \mathbb N \cup \{0\}.$ Define the graph $\Gamma(m, n, k, l, x, y)$ as follows: 
    
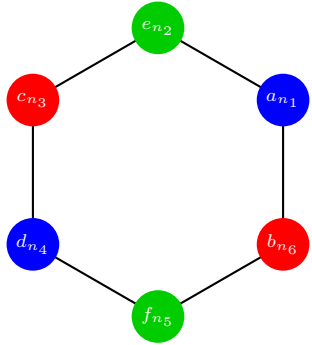
\begin{center}
\begin{tikzpicture}[
    scale=0.8,
    dot/.style={
        circle, 
        minimum size=7mm, 
        inner sep=0.5pt, 
        font=\tiny, 
        text=white,
        thick
    }
]

    \definecolor{green!80!black}{RGB}{0, 100, 0}

    \coordinate (a1) at (30:2.4);
    \coordinate (e1) at (90:2.4);
    \coordinate (c1) at (150:2.4);
    \coordinate (d1) at (210:2.4);
    \coordinate (f1) at (270:2.4);
    \coordinate (b1) at (330:2.4);

    \draw[thick] (a1) -- (e1) -- (c1) -- (d1) -- (f1) -- (b1) -- cycle;

    \node[dot, fill=blue] at (a1) {$a_{n_1}$};
    \node[dot, fill=green!80!black] at (e1) {$e_{n_2}$};
    \node[dot, fill=red] at (c1) {$c_{n_3}$};
    \node[dot, fill=blue] at (d1) {$d_{n_4}$};
    \node[dot, fill=green!80!black] at (f1) {$f_{n_5}$};
    \node[dot, fill=red] at (b1) {$b_{n_6}$};

\end{tikzpicture}
         
    \end{center}
    Let $V_e^r, V_f^r, V_e^b, V_f^b, V_e$, $V_f$ be disjoint sets of vertices that satisfies the following conditions:
    \begin{enumerate}[leftmargin=0.5cm]
        \item $V_e,$ $V_f$ are sets of all the green vertices such that $|V_e| = x, |V_f| = y$. These elements are labeled $e_i$ and $f_i$, respectively.
        \item $V_e^r$ is the set of red vertices that satisfies: 
        \begin{itemize}[leftmargin=0.5cm]
            \item $|V_e^r| = k$
            \item Each element is connected to:
            \begin{itemize}[leftmargin=0.5cm] 
                \item all elements in $V_e$ 
                \item all elements in $V_f^b$
            \end{itemize}
            \item These vertices will be labeled $c_i$
        \end{itemize} 
        \item $V_e^b$ is the set of blue vertices that satisfies: 
        \begin{itemize}[leftmargin=0.5cm]
            \item $|V_e^b| = n$
            \item Each element is connected to:
            \begin{itemize}[leftmargin=0.5cm]
                \item all elements in $V_e$
                \item all elements in $V_f^r$
            \end{itemize}
            \item These vertices will be labeled $a_i$
        \end{itemize}
        \item $V_f^r$ is the set of red vertices that satisfies: 
        \begin{itemize}[leftmargin=0.5cm]
            \item $|V_f^r| = m$
            \item Each element is connected to:
            \begin{itemize}[leftmargin=0.5cm]
                \item all elements in $V_f$
                \item all elements in $V_e^b$
            \end{itemize}
            \item These vertices will be labeled $b_i$
        \end{itemize}
        \item $V_f^b$ is the set of blue vertices that satisfies: 
        \begin{itemize}[leftmargin=0.5cm]
            \item $|V_f^b| = l$
            \item Each element is connected to:
            \begin{itemize}[leftmargin=0.5cm]
                \item all elements in $V_f$
                \item all elements in $V_e^r$
            \end{itemize}
            \item These vertices will be labeled $d_i$
        \end{itemize}
    \end{enumerate}
\end{definition}
\begin{lemma}\label{Gamma-bipartite}
    $\Gamma(m, n, k, l, x, y)$ is bipartite.
\end{lemma}
\begin{proof}
    We will color all the vertices in $V_e^r \cup V_e^b \cup V_f$ blue and $V_f^r \cup V_f^b \cup V_e$ red. By definition of each of the sets, we know that in the set of blue vertices, there is no edge between them. Similarly, there is no edge in the set of red vertices.

    Thus, the graph must be bipartite. See the graph below. 
    
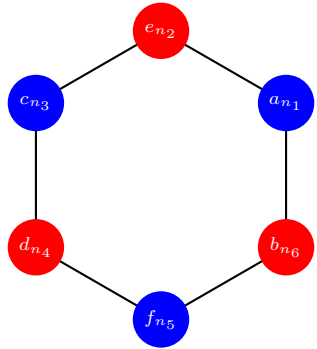
\begin{center}
\begin{tikzpicture}[
    scale=0.8,
    dot/.style={
        circle, 
        minimum size=7.5mm, 
        inner sep=0pt, 
        font=\tiny, 
        text=white,
        thick
    }
]

    \definecolor{green!80!black}{RGB}{0, 100, 0}

    \coordinate (a1) at (30:2.4);
    \coordinate (e1) at (90:2.4);
    \coordinate (c1) at (150:2.4);
    \coordinate (d1) at (210:2.4);
    \coordinate (f1) at (270:2.4);
    \coordinate (b1) at (330:2.4);

    \draw[thick] (a1) -- (e1) -- (c1) -- (d1) -- (f1) -- (b1) -- cycle;

    \node[dot, fill=blue] at (a1) {$a_{n_1}$};
    \node[dot, fill=red]  at (e1) {$e_{n_2}$};
    \node[dot, fill=blue] at (c1) {$c_{n_3}$};
    \node[dot, fill=red]  at (d1) {$d_{n_4}$};
    \node[dot, fill=blue] at (f1) {$f_{n_5}$};
    \node[dot, fill=red]  at (b1) {$b_{n_6}$};

\end{tikzpicture}
         
    \end{center}
\end{proof}
\begin{lemma}\label{nonzero}
    Let $\Gamma(m, n, 0, l, x, y)$ be a subgraph of an MPGC where $l > 0.$ Either: 
    \begin{itemize}[leftmargin=1cm]
        \item there exists another vertex in the MPGC which, when considered with the subgraph, creates a $\Gamma(m, n, 1, l, x, y)$ subgraph of the original MPGC, 
        \item or each edge in the subgraph is a part of a 5-cycle.
    \end{itemize}
\end{lemma}
\begin{proof}
    We will do this in a few steps:
    \begin{enumerate}
        \item show that there exists a vertex $c$ connected to both a vertex $d_i$ in $V_f^b$ and a vertex $e_j$ in $V_e$,
        \item show that if there exists $v\in V_e\cup V_f^b$ and $(c,v)$ is not an edge, then every edge in the subgraph is contained in a 5-cycle,
        \item and finish. 
    \end{enumerate}
    If $(d_{n_1}, e_{n_2})$ is in the MPGC, then each edge is contained in the 5-cycle $(d_{n_1}, e_{n_2}, a_{n_3}, b_{n_4}, f_{n_5}),$ which finishes the problem. Thus, there must exist indices $i, j$ such that a red vertex blocks the edge $(d_i, e_j)$. Let this vertex be $c.$ Note that $c\notin V_f^r$ due to triangle-freeness, so $c$ is a new vertex outside of our subgraph but in our MPGC. 
    
    Suppose that the edge $(c,d_k)$ does not exist for some index $k\neq i$. Let $g_k$ be the green vertex that blocks the edge. We will prove that if $g_k$ is not a new vertex, then each edge in the subgraph is in a 5-cycle.

    If $g_k = e_l$, we have $(e_l, d_k, f_{n_2}, b_{n_3}, a_{n_4})$ which covers all edges but $(d_{n_1}, f_{n_2})$ for $n_1\neq k$ and $(e_{n_1}, a_{n_2})$ for $n_1\neq l$. Note that $j$ and $l$ are not necessarily distinct.  
    
    
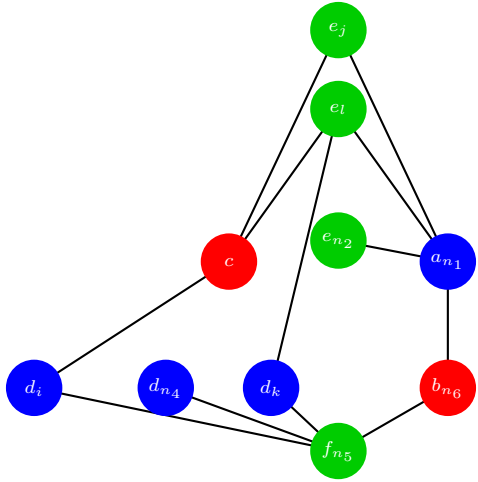
\begin{center}
\begin{tikzpicture}[
    scale=0.7,
    dot/.style={
        circle, 
        minimum size=7.5mm, 
        inner sep=0pt, 
        font=\tiny, 
        text=white,
        thick
    }
]

    \definecolor{green!80!black}{RGB}{0, 100, 0}

    \coordinate (a1) at (30:2.4);
    
    \coordinate (e1) at (90:1.6);
    
    \coordinate (c1) at (150:2.4);
    
    \coordinate (d1_base) at ($(210:2.4) + (0.8,0)$); 
    
    \coordinate (f1) at (270:2.4);
    \coordinate (b1) at (330:2.4);

    \coordinate (el) at ($(e1) + (0,2.5)$);
    \coordinate (ej) at ($(e1) + (0,4)$);
    \coordinate (dk) at (d1_base); 
    \coordinate (d1) at ($(d1_base) + (-2,0)$);
    \coordinate (di) at ($(d1_base) + (-4.5,0)$);

    \draw[thick] (e1) -- (a1) -- (b1) -- (f1) -- (d1);
    \draw[thick] (el) -- (a1) -- (ej);
    \draw[thick] (di) -- (f1) -- (dk);
    \draw[thick] (di) -- (c1) -- (ej);
    \draw[thick] (dk) -- (el) -- (c1);

    \node[dot, fill=blue] at (a1) {$a_{n_1}$};
    \node[dot, fill=green!80!black] at (e1) {$e_{n_2}$};
    \node[dot, fill=blue] at (d1) {$d_{n_4}$};
    \node[dot, fill=green!80!black] at (f1) {$f_{n_5}$};
    \node[dot, fill=red] at (b1) {$b_{n_6}$};
    \node[dot, fill=red] at (c1) {$c$};

    \node[dot, fill=blue] at (di) {$d_{i}$};
    \node[dot, fill=blue] at (dk) {$d_{k}$};

    \node[dot, fill=green!80!black] at (el) {$e_l$};
    \node[dot, fill=green!80!black] at (ej) {$e_j$};

\end{tikzpicture}
    \end{center}
    We prove that $(d_{m}, f_{n_1})$ is always contained in a 5-cycle via construction for some $m\neq k$. We do casework on if $(c, d_m)$ exists:
    \begin{itemize}
        \item If it exists, we have $(c, d_m, f_{n_1}, d_k, e_l).$
        \item If it does not exist, let the green vertex that blocks it be $g.$ If $g=e_{l'}$ for some $l',$ then we have $(d_m, e_{l'}, a_{n_1}, b_{n_2}, f_{n_3}).$ If $g = f_{l'}, $ we have a triangle $(c, f_{l'}, d_i).$ For the case when $g$ is a new point, we have $(c, g, d_{m}, f_{n_1}, d_i).$ 
    \end{itemize}
    Now, we prove that $(e_{m}, a_{n_1})$ is also contained in a 5-cycle for some $m\neq l$. We do casework on if $(c, e_m)$ exists:
    \begin{itemize}
        \item If it exists, we have $(c,e_m,a_{n_1},b_{n_2},f_{n_3},d_k,e_l)$. 
        \item If it does not exist, let the blue vertex that blocks it be $g.$ If $g = a_{l'}$ for some $l',$ then we have the triangle $(c,e_j,a_{l'}).$ We do not care if $g = d_{l'}$ for some $l'$ because the 5-cycle $(c,g,e_m,a_{n_1},e_l)$ works regardless.
    \end{itemize}
    Thus, if $g_k = e_l$ for some $l,$ every edge of the subgraph is contained in a 5-cycle. 

    If $g_k = f_l$ for some $l,$ we have the triangle $(c, f_l, d_i)$.

    Hence, suppose $g_k$ is distinct from vertices in our subgraph for every non-existing $(c,d_k)$ edge. Since $e$ and $d$ are symmetrical, we may assume that for every vertex that blocks the edge $(c, e_k)$ is not in our subgraph. 

    Suppose there existed a vertex $v\in V_{e}\cup V_f^b$ such that $(c,v)$ was not connected. 
    
    Then, if a vertex $u$ blocks the edge, by the above, we have $u$ is a new vertex. Note that due to the symmetry, we may assume $v\in V_f^b.$ Hence, $v$ is blue, $u$ is green. Observe $(c,u,v,f_{n_1}, b_{n_2}, a_{n_3}, e_j)$ which covers all edges but $(e_{n_1}, a_{n_2})$ for $n_1\neq j$ and $(d_{n_1}, f_{n_2})$. The proof is the same as above, but we state it for completeness. 
    
    We prove that $(d_{m}, f_{n_1})$ is always contained in a 5-cycle via construction for some $m\neq k$. We do casework on if $(c, d_m)$ exists:
    \begin{itemize}
        \item If it exists, we have $(c, d_m, f_{n_1}, v, u).$
        \item If it does not exist, let the green vertex that blocks it be $g.$ If $g=e_{l'}$ for some $l',$ then we have $(d_m, e_{l'}, a_{n_1}, b_{n_2}, f_{n_3}).$ If $g = f_{l'}, $ we have a triangle $(c, f_{l'}, d_i).$ For the case when $g$ is a new point, we have $(c, g, d_{m}, f_{n_2}, d_i).$ 
    \end{itemize}
    Now, we prove that $(e_{m}, a_{n_1})$ is also contained in a 5-cycle. We do casework on if $(c, e_m)$ exists for some $m\neq l$:
    \begin{itemize}
        \item If it exists, we have $(c,e_m,a_{n_1},b_{n_2},f_{n_3},v,u)$. 
        \item If it does not exist, let the blue vertex that blocks it be $g.$ If $g = a_{l'}$ for some $l',$ then we have the triangle $(c,e_j,a_{l'}).$ If $g = d_{l'}$ for some $l'$ or if $g$ is a new vertex, we have the 5-cycle $(c,g,e_m,a_{n_1},e_j)$.
    \end{itemize}
    Thus, each edge is contained in a 5-cycle if $v$ exists. If $v$ does not exist, then due to minimally triangle-freeness condition, $(c,e_{n_1})$ and $(c,d_{n_1})$ are in the graph. It follows that considering $c$ as a part of our subgraph, we get a $\Gamma(m,n,1,l,x,y)$ graph. 
\end{proof}
\begin{lemma}\label{counterexample}
    Let $\Gamma$ be an MPGC with a $\Gamma(m, n, k, l, x, y)$ subgraph. Then, each edge in the $\Gamma(m, n, k, l, x, y)$ subgraph is contained in a 5-cycle.
\end{lemma}
\begin{proof}
    We will do this in a few steps:
    \begin{enumerate}
        \item reduce the case to either $|V_e^r| = |V_f^b| = 0$ or both set's cardinality is non-zero,
        \item show that there must exist another vertex, $c$, connected to the $\Gamma(m, n, k, l, x, y)$ subgraph,
        \item we will do casework on which vertex $c$ is connected to show that either each edge in the $\Gamma(m, n, k, l, x, y)$ subgraph is contained in a 5-cycle, or else when considering $c,$ our subgraph becomes $\Gamma(m', n', k', l', x', y')$,
        \item and then finish. 
    \end{enumerate}
    We can assume that either $|V_e^r| = |V_f^b| = 0$ or both set's cardinality is non-zero. We can do this since Lemma \ref{nonzero} handles the case in which only one of the set's cardinalities is 0. Thus, we will also write 5 or 7-cycles that do not contain $c_{n_1}, d_{n_2}$ in the case that the sets are empty.

    By \cite{gruber} and Lemma \ref{Gamma-bipartite}, it follows that $\Gamma(m, n, k, l, x, y)$ can never be an MPG. Clearly, no edge can be added to $\Gamma(m, n, k, l, x, y)$ without violating one of the MPG conditions. Hence, there must exist another vertex. We will show that adding this vertex $\Gamma(m, n, k, l, x, y)$ will result in:
    \begin{itemize}
        \item another $\Gamma(m', n', k', l', x', y')$ graph with $\bullet'\geq \bullet$ and $\sum\bullet' > \sum \bullet$ where $\bullet$ are the inputs,
        \item or each edge being in a 5-cycle.
    \end{itemize}
    This will prove our desired result since MPGCs are finite graphs, so we can not add in infinitely many vertices. 

    Let our other vertex be $c.$ By Lemma \ref{connected}, we can assume that $c$ is connected to this subgraph. 
    \begin{lemma}
        If $c$ is blue, then each edge is contained in a 5-cycle, or it will produce another $\Gamma(m, n, k, l, x, y)$ graph. 
    \end{lemma}
    \begin{proof}
        Define $V_c, V_{nc}$ such that $V_c\cup V_{nc} = V_f^b$ and $V_c$ consists of the vertices connected to $c$ while $V_{nc}$ consists of the vertices without an edge to $c$. We first handle the case where $V_c, V_{nc}\neq \emptyset.$ 
        
        Suppose $b_i\in V_c, b_j\in V_{nc}.$ Then, define $v$ as the green vertex that blocks $(b_j,c)$. If $v\notin V_f$, $(c,b_i,v)$ is a triangle and if $v\in V_e,$ $(v,a_{n_1},b_j)$ is a triangle. Thus, $v$ is a new vertex. Observe the 7-cycles $(c,v,b_j,f_{n_1},d_{n_2},c_{n_3},e_{n_4})$ and $(c,v,b_j,f_{n_1},b_i,a_{n_2},e_{n_3})$ which covers all of the edges. Thus, if $V_c, V_{nc}\neq\emptyset,$ then each edge is covered in a 5-cycle. 
        
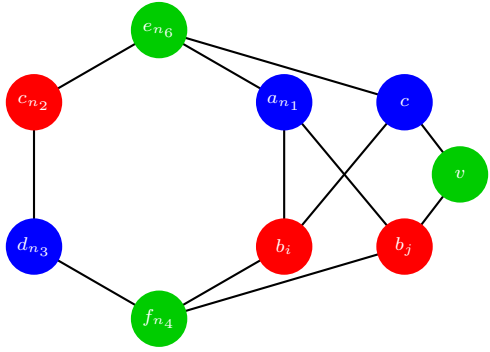
\begin{center}
\begin{tikzpicture}[
    scale=0.8,
    dot/.style={
        circle, 
        minimum size=7.5mm, 
        inner sep=0pt, 
        font=\tiny, 
        text=white,
        thick
    }
]

    \definecolor{green!80!black}{RGB}{0, 100, 0}

    \coordinate (a1) at (30:2.4);
    \coordinate (e1) at (90:2.4);
    \coordinate (c1) at (150:2.4);
    \coordinate (d1) at (210:2.4);
    \coordinate (f1) at (270:2.4);
    \coordinate (bi) at (330:2.4);

    \begin{scope}[shift={(bi)}]
        \coordinate (bj) at (2,0);
    \end{scope}
    
    \begin{scope}[shift={(a1)}]
        \coordinate (c) at (2,0);
    \end{scope}
    
    \coordinate (v) at (5,0);

    \draw[thick] (e1) -- (c) -- (v) -- (bj) -- (f1) -- (d1) -- (c1) -- (e1);
    
    \draw[thick] (e1) -- (a1) -- (bi) -- (f1);
    \draw[thick] (bi) -- (c);
    \draw[thick] (a1) -- (bj);

    \node[dot, fill=blue] at (a1) {$a_{n_1}$};
    \node[dot, fill=red] at (c1) {$c_{n_2}$};
    \node[dot, fill=blue] at (d1) {$d_{n_3}$};
    \node[dot, fill=green!80!black] at (f1) {$f_{n_4}$};
    \node[dot, fill=green!80!black] at (e1) {$e_{n_6}$};
    \node[dot, fill=red] at (bi) {$b_i$};
    \node[dot, fill=red] at (bj) {$b_j$};
    \node[dot, fill=blue] at (c) {$c$};
    \node[dot, fill=green!80!black] at (v) {$v$};

\end{tikzpicture}
        \end{center}
        
        Now, consider if $V_{nc} = \emptyset,$ and thus $c$ is connected to every vertex in $V_f^r.$ If $c$ was connected to every vertex in $V_e,$ we get $\Gamma(m+1,n,k,l,x,y)$. Hence, suppose $c$ is not connected to all vertices in $V_f$. 
        
        Suppose $(c,e_i)$ is an edge and $(c,e_j)$ is not an edge. Note that $e_j$ necessarily exists but $e_i$ does not have to necessarily exist. Let $(c,e_j)$ be blocked by $v$. We know $v$ is red. If $v \in V_e^r$, then $(v,e_i,c)$ is a triangle. If $v\in V_f^r,$ then $(v,e_j,a_{n_1})$ is a triangle. Thus, $v$ is a new vertex. Notice $(v,e_j,c_{n_1},d_{n_2},f_{n_3},b_{n_4},c)$ and $(c,v,e_j,c_{n_1},e_i)$ which covers all of the edges connected to at least one of the vertices $d_{n_1},c_{n_2}.$ We also have $(c,v,e_j,a_{n_1},b_{n_2})$ and $(c,v,e_j,a_{n_1},e_i)$ which covers all edges but $(b_{n_1},f_{n_2})$.
        
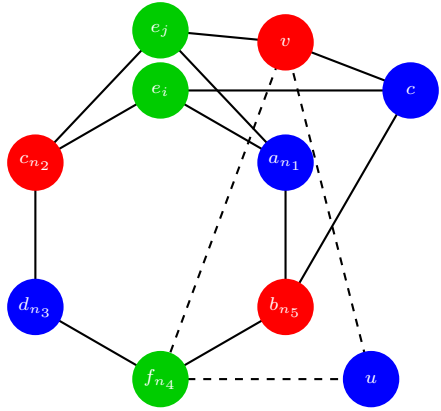
\begin{center}
\begin{tikzpicture}[
    scale=0.8,
    dot/.style={
        circle, 
        minimum size=7.5mm, 
        inner sep=0pt, 
        font=\tiny, 
        text=white,
        thick
    }
]

    \definecolor{green!80!black}{RGB}{0, 100, 0}

    \coordinate (a1) at (30:2.4);
    \coordinate (c1) at (150:2.4);
    \coordinate (d1) at (210:2.4);
    \coordinate (f1) at (270:2.4);
    \coordinate (b1) at (330:2.4);

    \coordinate (ei) at (90:2.4);
    
    \begin{scope}[shift={(ei)}]
        \coordinate (ej) at (0,1);
    \end{scope}

    \coordinate (c) at (30:4.8);

    \begin{scope}[shift={(f1)}]
        \coordinate (u) at (3.5,0);
    \end{scope}

    \begin{scope}[shift={(a1)}]
        \coordinate (v) at (0,2);
    \end{scope}

    \draw[thick] (a1) -- (b1) -- (f1) -- (d1) -- (c1);
    \draw[thick] (a1) -- (ei) -- (c1);
    \draw[thick] (a1) -- (ej) -- (c1);
    \draw[thick] (ej) -- (v) -- (c);
    \draw[thick] (ei) -- (c) -- (b1);
    
    \draw[thick, dashed] (v) -- (u) -- (f1) -- cycle;

    \node[dot, fill=blue] at (a1) {$a_{n_1}$};
    \node[dot, fill=red] at (c1) {$c_{n_2}$};
    \node[dot, fill=blue] at (d1) {$d_{n_3}$};
    \node[dot, fill=green!80!black] at (f1) {$f_{n_4}$};
    \node[dot, fill=red] at (b1) {$b_{n_5}$};
    \node[dot, fill=green!80!black] at (ei) {$e_i$};
    \node[dot, fill=green!80!black] at (ej) {$e_j$};
    \node[dot, fill=blue] at (c) {$c$};
    \node[dot, fill=blue] at (u) {$u$};
    \node[dot, fill=red] at (v) {$v$};

\end{tikzpicture}
        \end{center}
        Consider $(f_k,v)$:
        \begin{itemize}
            \item If $(f_k,v)$ is connected, take $(f_k,b_{n_1},a_{n_2},e_j,v)$.
            \item If $(f_k,v)$ is not connected, let $u$ block it. If $u = c$, then $(c,b_{n_1},f_{k})$ is a triangle. If $u \in V_e^b,$ then $(u,f_{k},b_{n_1})$ is a triangle. Otherwise, take $(u,f_{k},b_{n_1},c,v)$.
        \end{itemize}
        Thus, if $V_{nc}=\emptyset,$ then our conclusion holds. Finally, we only have to check when $c$ is not connected to $b_{n_1}$. By symmetry, we may assume $c$ is not connected to $c_{n_1}$ as well. Hence, suppose $c$ is connected to at least one of the vertices in $V_f$ or $V_e.$ 
        
        Suppose $c$ is not connected to any vertices in $V_f.$ Hence, suppose $(c,e_k)$ is an edge. Let $v_i$ block $(c,b_i)$ for all $i$. Note that $v_i$ is green. If $v_i\in V_e,$ then $(v_i,a_{n_1},b_i)$ is a triangle. If $v_i\in V_f,$ then this contradicts our assumption. Thus, $v_i$ are new, not necessarily distinct, vertices. Let $u_i$ block $(c,f_i)$ for all $i$. Note that $u_i$ is red. If $u_i\in V_f^r,$ and $u_i = b_{j}$ for some $j,$ then $(b_j,c,v_j)$ is a triangle. If $u_i\in V_e^r,$ then $(f_i,d_{n_1},u_i)$ is a triangle. Hence, $u$ is also a new vertex. 
        
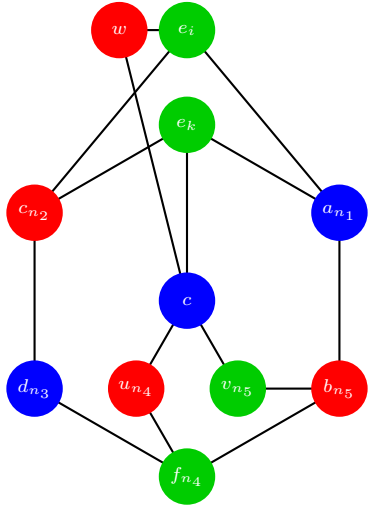
\begin{center}
\begin{tikzpicture}[
    scale=0.9,
    dot/.style={
        circle, 
        minimum size=7.5mm, 
        inner sep=0pt, 
        font=\tiny, 
        text=white,
        thick
    }
]

    \definecolor{green!80!black}{RGB}{0, 100, 0}

    \coordinate (a1) at (30:2.6);
    \coordinate (c1) at (150:2.6);
    \coordinate (d1) at (210:2.6);
    \coordinate (f1) at (270:2.6);
    \coordinate (b1) at (330:2.6);
    \coordinate (ei) at (90:2.6);
    \coordinate (c) at (0,0);

    \begin{scope}[shift={(ei)}]
        \coordinate (ej) at (0,1.4);
    \end{scope}

    \begin{scope}[shift={(ej)}]
        \coordinate (w) at (-1,0);
    \end{scope}

    \coordinate (u) at ($1/3*(f1) + 1/3*(c) + 1/3*(d1)$);

    \coordinate (v) at ($1/3*(f1) + 1/3*(b1) + 1/3*(c)$);

    \draw[thick] (a1) -- (b1) -- (f1) -- (d1) -- (c1) -- (ei) -- (a1);
    \draw[thick] (a1) -- (ej) -- (c1);
    \draw[thick] (ej) -- (w) -- (c);
    \draw[thick] (ei) -- (c);
    \draw[thick] (c) -- (v) -- (b1);
    \draw[thick] (c) -- (u) -- (f1);

    \node[dot, fill=blue] at (a1) {$a_{n_1}$};
    \node[dot, fill=red] at (c1) {$c_{n_2}$};
    \node[dot, fill=blue] at (d1) {$d_{n_3}$};
    \node[dot, fill=green!80!black] at (f1) {$f_{n_4}$};
    \node[dot, fill=red] at (b1) {$b_{n_5}$};
    \node[dot, fill=green!80!black] at (ei) {$e_k$};
    \node[dot, fill=green!80!black] at (ej) {$e_i$};
    \node[dot, fill=blue] at (c) {$c$};
    \node[dot, fill=red] at (u) {$u_{n_4}$};
    \node[dot, fill=green!80!black] at (v) {$v_{n_5}$};
    \node[dot, fill=red] at (w) {$w$};

\end{tikzpicture}
        \end{center}
        Now, observe the cycles $(c,u_{n_1},f_{n_1},b_{n_2},v_{n_2})$, $(c,v_{n_1},b_{n_1},a_{n_2},e_k)$, $(e_k, c,v_{n_1},b_{n_1},f_{n_2},d_{n_2},c_{n_3})$. These cycles cover all edges but the vertices in $V_e$ that are not connected to $c.$ Let $w$ block $(e_i,c).$ If $w\in V_e^b,$ then $(e_k,c,w)$ is a triangle. Otherwise, the 5-cycle $(e_j,w,c,e_k,a_{n_1})$ finishes. 

        Now, suppose that $c$ is connected to some but not all vertices in $V_f.$ Suppose that $(f_i,c)$ was connected and $(f_j,c)$ was blocked by a red vertex $v.$ If $v\in V_f^r,$ then $(f_j,d_{n_1},v)$ is a triangle. If $v\in V_e^r,$ then $(v, f_i,b_{n_1})$ is a triangle. Now, consider the edge $(e_k,c)$:
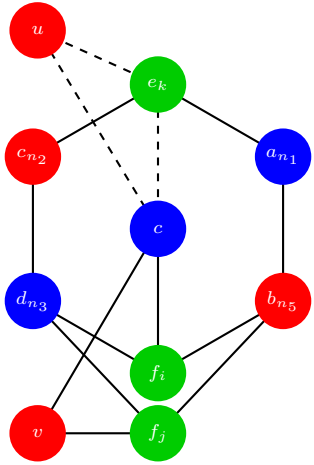
        \begin{center}
\begin{tikzpicture}[
    scale=0.8,
    dot/.style={
        circle, 
        minimum size=7.5mm, 
        inner sep=0pt, 
        font=\tiny, 
        text=white,
        thick
    }
]

    \definecolor{green!80!black}{RGB}{0, 100, 0}

    \coordinate (a1) at (30:2.4);
    \coordinate (c1) at (150:2.4);
    \coordinate (d1) at (210:2.4);
    \coordinate (b1) at (330:2.4);
    \coordinate (e1) at (90:2.4);
    \coordinate (fi) at (270:2.4);
    \coordinate (c) at (0,0);

    \begin{scope}[shift={(fi)}]
        \coordinate (fj) at (0,-1);
    \end{scope}

    \begin{scope}[shift={(fj)}]
        \coordinate (v) at (-2,0);
    \end{scope}

    \begin{scope}[shift={(e1)}]
        \coordinate (u) at (-2, 0.9);
    \end{scope}

    \draw[thick] (a1) -- (b1) -- (fi) -- (d1) -- (c1) -- (e1) -- (a1);
    \draw[thick] (d1) -- (fj) -- (b1);
    \draw[thick] (fj) -- (v) -- (c);
    \draw[thick] (c) -- (fi);

    \draw[thick, dashed] (c) -- (e1) -- (u) -- cycle;

    \node[dot, fill=blue] at (a1) {$a_{n_1}$};
    \node[dot, fill=red] at (c1) {$c_{n_2}$};
    \node[dot, fill=blue] at (d1) {$d_{n_3}$};
    \node[dot, fill=red] at (b1) {$b_{n_5}$};
    \node[dot, fill=green!80!black] at (fi) {$f_{i}$};
    \node[dot, fill=green!80!black] at (fj) {$f_{j}$};
    \node[dot, fill=green!80!black] at (e1) {$e_k$};
    \node[dot, fill=blue] at (c) {$c$};
    \node[dot, fill=red] at (u) {$u$};
    \node[dot, fill=red] at (v) {$v$};

\end{tikzpicture}
        \end{center}
        \begin{itemize}[leftmargin=0.5cm]
            \item If $(e_k,c)$ is connected, take $(c,f_i,b_{n_1},a_{n_2},e_k)$, $(c,f_i,b_{n_1},f_j,v)$, and $(e_k,c,f_i,d_{n_1},c_{n_2})$ cover all of the edges.  
            \item If $(e_k,c)$ is not connected, suppose the red vertex $u$ blocks it. If $u\in V_f^r,$ then $(e_k,a_{n_1},u)$ is a triangle. If $u\in V_e^r$, we do casework on the size of $V_e^r.$ 
            \begin{itemize}[leftmargin=0.5cm]
                \item If $V_e^r \neq \{u\}$ We have $(e_k,u,c,v,f_j,b_{n_1},a_{n_1})$, $(c,v,f_j,b_{n_1},f_i)$, $(c,v,f_j,d_{n_1},e_i)$, $(f_j,v,c,u,d_{n_1})$, and $(e_k,u,c,v,f_j,d_{n_1},c_{n_2})$ cover all of the edges. 
                \item If $V_e^r = \{u\}$, then we do casework on $(e_k, v).$ If it is connected, $(e_k,u,d_{n_1},f_j,v).$ If it is not connected, let the blue vertex $w$ block it. If $w = c$, then $(c,e_k,c_{n_1})$ is a triangle. Otherwise, we have $(c,v,w,e_k,u)$ and $(c,c_{n_1},d_{n_2},f_j,v).$ 
            \end{itemize}
            Now, suppose $u$ is a new vertex. Then, the cycles in the above case when $u\in V_e^r$ finish. 
        \end{itemize}
        This covers all of the edges, so if $c$ is not connected to some edge in $V_f,$ we are done. Hence, suppose $c$ is connected to every edge in $V_f.$ 
        
        Now, suppose the green vertex $v$ blocks $(c,c_i)$ for some index $i.$ If $v\in V_f,$ $(c_i,d_{n_1},v)$ is a triangle. If $v\in V_e,$ then we have $(c,f_{n_1},d_{n_2}, c_{n_3}, v),$ $(c,f_{n_1},d_{n_2}, c_{n_3}, e_{n_4},a_{n_5},v)$, $(c,f_{n_1},b_{n_2}, a_{n_3}, v)$. If $v$ is a new vertex, take $(c,v,c_{i},d_{n_1},f_{n_2})$, $(c,v,c_{i},e_{n_1},a_{n_2},b_{n_3},f_{n_4})$, and $(c,v,c_i,e_{n_1},c_{n_2},d_{n_3},f_{n_4}).$ 

        Hence, suppose that $(c,c_{n_1})$ is an edge. The graph is now $\Gamma(m,n,k,l+1,x,y).$
    \end{proof}
    \begin{lemma}
        If $c$ is red, then each edge is contained in a 5-cycle, or it will produce another $\Gamma(n, m, k, l, x, y)$ graph.  
    \end{lemma}
    \begin{proof}
        It has the same proof as the previous Lemma.
    \end{proof}
    \begin{lemma}
        If $c$ is green, then each edge is contained in a 5-cycle, or it will produce another $\Gamma(n, m, k, l, x, y)$ graph. 
    \end{lemma}
    \begin{proof}
        Without loss of generality, let $c$ be connected to the red vertex $b_1.$ There exist 2 cases:
        \begin{enumerate}[leftmargin=0.5cm]
            \item if there exists $j$ such that $c$ is not connected to $b_j,$  
            \item and if $c$ is connected to each $b_{n_1}.$
        \end{enumerate}
        
        We handle the first case first. Suppose that $(b_i, c)$ are connected and $(b_j,c)$ are blocked by a blue vertex $d$. We know that $b_i, b_j$ necessarily exist. If $d\in V_e^b$,  then $(c,b_{i},d)$ is a triangle. If $d\in V_f^b,$ then $(b_j,d,f_{n_1})$ is a triangle. Hence, $d$ is a new vertex. 

        First note that $(c,d,b_j,a_{n_1}, b_i)$ and $(c,d,b_j,f_{n_1},b_i)$ covers the edges containing the vertex $b_{n_1}.$ We will do casework on if the edge $(e_k,d)$ is connected. 
        
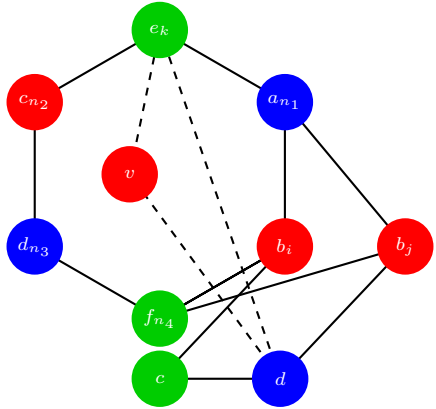
\begin{center}
\begin{tikzpicture}[
    scale=0.8,
    dot/.style={
        circle, 
        minimum size=7.5mm, 
        inner sep=0pt, 
        font=\tiny, 
        text=white,
        thick
    }
]

    \definecolor{green!80!black}{RGB}{0, 100, 0}

    \coordinate (a1) at (30:2.4);
    \coordinate (c1) at (150:2.4);
    \coordinate (d1) at (210:2.4);
    \coordinate (e1) at (90:2.4);
    \coordinate (f1) at (270:2.4);
    \coordinate (bi) at (330:2.4);
    \coordinate (v) at (-.5,0);

    \begin{scope}[shift={(bi)}]
        \coordinate (bj) at (2,0);
    \end{scope}

    \begin{scope}[shift={(f1)}]
        \coordinate (c) at (0,-1);
    \end{scope}

    \begin{scope}[shift={(c)}]
        \coordinate (d) at (2,0);
    \end{scope}

    \draw[thick] (bi) -- (f1) -- (d1) -- (c1) -- (e1) -- (a1) -- (bj) -- (f1) -- (bi) -- (a1);
    \draw[thick] (bj) -- (d) -- (c) -- (bi);

    \draw[thick, dashed] (d) -- (e1) -- (v) -- cycle;

    \node[dot, fill=blue] at (a1) {$a_{n_1}$};
    \node[dot, fill=red] at (c1) {$c_{n_2}$};
    \node[dot, fill=blue] at (d1) {$d_{n_3}$};
    \node[dot, fill=green!80!black] at (f1) {$f_{n_4}$};
    \node[dot, fill=green!80!black] at (e1) {$e_k$};
    \node[dot, fill=green!80!black] at (c) {$c$};
    \node[dot, fill=blue] at (d) {$d$};
    \node[dot, fill=red] at (v) {$v$};
    \node[dot, fill=red] at (bi) {$b_{i}$};
    \node[dot, fill=red] at (bj) {$b_{j}$};

\end{tikzpicture}
        \end{center}
        \begin{itemize}
            \item If they are connected, $(c,d,e_k,a_{n_1},b_{i})$ and $(c,d,e_k,c_{n_1},d_{n_2},f_{n_3},b_i)$ finish.
            \item If they are not connected, let the red vertex $v$ block it. If $v\in V_f^r,$ then $(e_k,v,a_{n_1})$ is a triangle. If $v\in V_e^r,$ then $(e_k,v,d,b_j,a_{n_1})$, $(v,d_{n_1},f_{n_2},b_j,d)$, and $(e_k,v,d,b_j,f_{n_1},d_{n_2},c_{n_3})$ finish. If $v$ is a new vertex, then $(e_k,v,d,b_j,a_{n_1})$ and $(d,v,e_k,c_{n_1},d_{n_2},f_{n_3},b_j)$ cover the edges. 
        \end{itemize}
        Hence, in the first case, each edge is contained in a 5-cycle. Now consider the 2nd case. We have that $c$ is connected to all of $b_i.$ 
        
        If $|V_f^b| = 0,$ then we have a $\Gamma(n,m,0,0,x,y+1)$ graph. Hence, suppose $k, l > 0.$ We know that $c$ is connected to every vertex in $V_f^r.$ Suppose $(c,d_i)$ is not an edge, and is blocked by the red vertex $v$. If $v\in V_f^r,$ then $(v,d_i,f_{n_1})$ is a triangle. If $v\in V_e^r$, then $(v,c,b_{n_1},a_{n_2},e_{n_3})$, $(v,c,b_{n_1},f_{n_2},d_{n_3})$, and $(v,c,b_{n_1},a_{n_2},e_{n_3}, c_{n_4}, d_{n_5})$ cover the edges. 
        
        Hence, suppose $v$ is a new vertex. Thus, we have $(d_i, f_{n_1}, b_{n_2}, c, v)$, $(d_i, v, c, b_{n_1}, a_{n_2}, e_{n_3}, c_{n_4})$ which cover all edges but the ones connected to $d_{n_1}$ for $n_1\neq j$. Observe $(d_i,v,c,b_{n_1},f_{n_2},d_{n_3},c_{n_4})$ which covers all of the edges. 
        
        Finally, if $(c,d_{n_1})$ is an edge, then we get the graph $\Gamma(n,m,k,l,x,y+1).$ 
    \end{proof}
   Since adding any colored vertex will result in each edge of the subgraph being in a 5-cycle, or generate a new $\Gamma(m, n, k, l, x, y)$, we are done. 
\end{proof}
We will now prove that each edge in a 4-cycle is part of a 5-cycle. Note that there are only two non-isomorphic 3-colorings of a 4-cycle: 
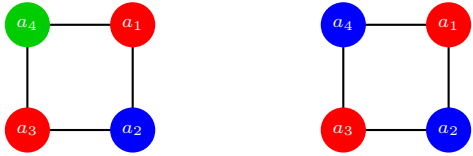
\begin{center}
\begin{tikzpicture}[
    scale=0.7,
    dot/.style={
        circle, 
        minimum size=6mm, 
        inner sep=0pt, 
        font=\tiny, 
        text=white,
        thick
    }
]

    \definecolor{green!80!black}{RGB}{0, 100, 0}

    \begin{scope}[xshift=-1.5cm]
        \coordinate (A1) at (-0.5, 1);
        \coordinate (B1) at (-0.5, -1);
        \coordinate (D1) at (-2.5, -1);
        \coordinate (C1) at (-2.5, 1);

        \draw[thick] (A1) -- (B1) -- (D1) -- (C1) -- cycle;

        \node[dot, fill=red] at (A1) {$a_1$};
        \node[dot, fill=blue] at (B1) {$a_2$};
        \node[dot, fill=red] at (D1) {$a_3$};
        \node[dot, fill=green!80!black] at (C1) {$a_4$};
    \end{scope}

    \begin{scope}[xshift=1.5cm]
        \coordinate (A2) at (2.5, 1);
        \coordinate (B2) at (2.5, -1);
        \coordinate (D2) at (0.5, -1);
        \coordinate (C2) at (0.5, 1);

        \draw[thick] (A2) -- (B2) -- (D2) -- (C2) -- cycle;

        \node[dot, fill=red] at (A2) {$a_1$};
        \node[dot, fill=blue] at (B2) {$a_2$};
        \node[dot, fill=red] at (D2) {$a_3$};
        \node[dot, fill=blue] at (C2) {$a_4$};
    \end{scope}

\end{tikzpicture}
\end{center}

Define $V_a = (a_1, a_3),$ and $V_b = (a_2, a_4).$ Let $V, E$ be the vertex and edge set of our MPGC.
\begin{lemma} \label{1st4}
    If the first coloring of a 4-cycle is a subgraph of a MPGC, then each edge is part of a 5-cycle. 
\end{lemma}
\begin{proof}
    Note that the complement is disconnected. Let $c$ be a vertex such that it forms a path between $V_a$ and $V_b.$ $c$ must exist because of \cite[Corollary 4.4]{florez}. 

    We claim that $c$ is connected to at most 1 of the 4 from the original graph. $c$ can not be connected to both vertices in either $V_a$ or $V_b,$ or else it contradicts the fact that $c$ is connected to them in the original graph. $c$ also can not be connected to one from each because it will form a triangle. Hence, from here, we will assume that for every vertex $c$ is connected to, the vertex is not a part of the 4-cycle, unless stated otherwise.
    
    Hence, there are 2 cases, if $c$ is connected to 1 of the vertices, or if $c$ is connected to none. 
    \begin{enumerate}[leftmargin=0.5cm]
        \item If $c$ is connected to 1 vertex, we will do casework on which color vertex it is connected to. 
        \begin{enumerate}[leftmargin=0.5cm]
            \item If $c$ is connected to a red vertex, we first note that $c$ is not red. Without loss of generality, we assume that $(c,a_3)$ is connected. Hence, there must exist another vertex $d$ that will create a triangle if $c, a_1$ is connected. Take $(c,d,a_1,a_4,a_3)$ and $(c,d,a_1,a_2,a_3)$ and we get each edge is covered. 
            
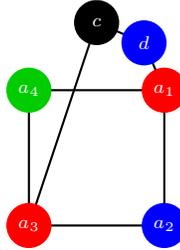
\begin{center}
\begin{tikzpicture}[
    scale=0.9,
    dot/.style={
        circle, 
        minimum size=6mm, 
        inner sep=0pt, 
        font=\tiny, 
        text=white,
        thick
    }
]

    \definecolor{green!80!black}{RGB}{0, 100, 0}

    \coordinate (a1) at (1, 1);
    \coordinate (a2) at (1, -1);
    \coordinate (a3) at (-1, -1);
    \coordinate (a4) at (-1, 1);
    
    \coordinate (c) at (0, 2);
    \coordinate (d) at (0.7, 1.7);

    \draw[thick] (a1) -- (a2) -- (a3) -- (a4) -- cycle;
    
    \draw[thick] (a3) -- (c) -- (d) -- (a1);

    \node[dot, fill=red] at (a1) {$a_1$};
    \node[dot, fill=blue] at (a2) {$a_2$};
    \node[dot, fill=red] at (a3) {$a_3$};
    \node[dot, fill=green!80!black] at (a4) {$a_4$};
    
    \node[dot, fill=black] at (c) {$c$};
    \node[dot, fill=blue] at (d) {$d$};

\end{tikzpicture}
            \end{center}
            \item If $c$ is connected to the green/blue, Without loss of generality, let $c$ be connected to the green vertex. There exists 2 cases, if $c$ is red, and if $c$ is blue. 
            \begin{enumerate}[leftmargin=0.5cm]
                \item If $c$ is red, there exists a vertex $c_2$ that is connected to both $a_2$ and $c$ to block the edge $(c, a_2)$(recall $c$ can not be connected to both vertices in $V_b$). Take $(c,a_4,a_3,a_2,c_2)$ and $(c,a_4,a_1,a_2,c_2)$ and each edge is covered. 
                
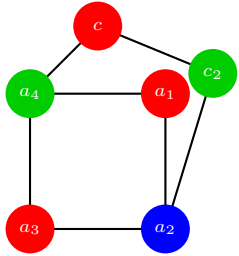
\begin{center}
\begin{tikzpicture}[
    scale=0.9,
    dot/.style={
        circle, 
        minimum size=6.5mm, 
        inner sep=0pt, 
        font=\tiny, 
        text=white,
        thick
    }
]

    \definecolor{green!80!black}{RGB}{0, 100, 0}

    \coordinate (a1) at (1, 1);
    \coordinate (a2) at (1, -1);
    \coordinate (a3) at (-1, -1);
    \coordinate (a4) at (-1, 1);
    
    \coordinate (c) at (0, 2);
    \coordinate (d) at (1.7, 1.3);

    \draw[thick] (a1) -- (a2) -- (a3) -- (a4) -- cycle;
    
    \draw[thick] (a4) -- (c) -- (d) -- (a2);

    \node[dot, fill=red] at (a1) {$a_1$};
    \node[dot, fill=blue] at (a2) {$a_2$};
    \node[dot, fill=red] at (a3) {$a_3$};
    \node[dot, fill=green!80!black] at (a4) {$a_4$};
    
    \node[dot, fill=red] at (c) {$c$};
    \node[dot, fill=green!80!black] at (d) {$c_2$};

\end{tikzpicture}
                \end{center}
                \item If $c$ is blue, then observe that this graph is isomorphic to $\Gamma(2, 1, 0, 0, 1, 1)$, so by Lemma \ref{counterexample}, we know that each edge is contained in a 5-cycle. 
                
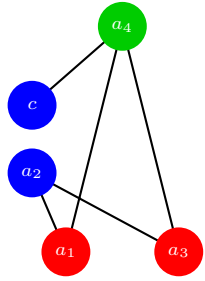
\begin{center}
\begin{tikzpicture}[
    scale=1.5,
    dot/.style={
        circle, 
        minimum size=6.5mm, 
        inner sep=0pt, 
        font=\tiny, 
        text=white,
        thick
    }
]

    \definecolor{green!80!black}{RGB}{0, 100, 0}

    \coordinate (a1) at (0, -1);
    \coordinate (a2) at (-0.3, -0.3);
    \coordinate (a3) at (1, -1);
    \coordinate (a4) at (0.5, 1);
    \coordinate (c) at (-0.3, 0.3);

    \draw[thick] (a1) -- (a2) -- (a3) -- (a4) -- cycle;
    
    \draw[thick] (a4) -- (c);

    \node[dot, fill=red] at (a1) {$a_1$};
    \node[dot, fill=blue] at (a2) {$a_2$};
    \node[dot, fill=red] at (a3) {$a_3$};
    \node[dot, fill=green!80!black] at (a4) {$a_4$};
    \node[dot, fill=blue] at (c) {$c$};

\end{tikzpicture}
                \end{center} 
            \end{enumerate}
        \end{enumerate}
        
        \item If $c$ is connected to none of the vertices, then there are 2 cases. If $c$ is red, or if $c$ is blue/green. 
        \begin{enumerate}[leftmargin=0.5cm]
            \item If $c$ is red, then there exists vertices $d_2, d_4$ connected to $a_2, a_4$ and to $c$ to block the edges($(c, a_2),(c, a_4)$). 
            
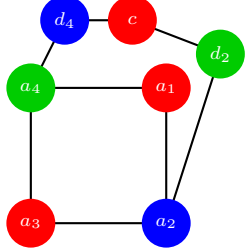
\begin{center}
\begin{tikzpicture}[
    scale=0.9,
    dot/.style={
        circle, 
        minimum size=6.5mm, 
        inner sep=0pt, 
        font=\tiny, 
        text=white,
        thick
    }
]

    \definecolor{green!80!black}{RGB}{0, 100, 0}

    \coordinate (a1) at (1, 1);
    \coordinate (a2) at (1, -1);
    \coordinate (a3) at (-1, -1);
    \coordinate (a4) at (-1, 1);
    
    \coordinate (c) at (0.5, 2);
    \coordinate (d1) at (-0.5, 2);
    \coordinate (d2) at (1.8, 1.5);

    \draw[thick] (a1) -- (a2) -- (a3) -- (a4) -- cycle;
    
    \draw[thick] (a4) -- (d1) -- (c) -- (d2) -- (a2);

    \node[dot, fill=red] at (a1) {$a_1$};
    \node[dot, fill=blue] at (a2) {$a_2$};
    \node[dot, fill=red] at (a3) {$a_3$};
    \node[dot, fill=green!80!black] at (a4) {$a_4$};
    
    \node[dot, fill=red] at (c) {$c$};
    \node[dot, fill=blue] at (d1) {$d_4$};
    \node[dot, fill=green!80!black] at (d2) {$d_2$};

\end{tikzpicture}
            \end{center}
            However, this graph is isomorphic to $\Gamma(2, 1, 1, 1, 1, 1),$ and hence each edge is contained in a 5-cycle by Lemma \ref{counterexample}.
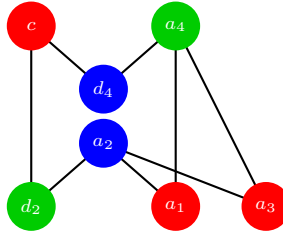
            \begin{center}
\begin{tikzpicture}[
    scale=1.2,
    dot/.style={
        circle, 
        minimum size=6.5mm, 
        inner sep=0pt, 
        font=\tiny, 
        text=white,
        thick
    }
]

    \definecolor{green!80!black}{RGB}{0, 100, 0}

    \coordinate (a1) at (0.5, -1);
    \coordinate (a2) at (-0.3, -0.3);
    \coordinate (a3_coord) at (-1.1, 1); 
    \coordinate (a4) at (0.5, 1);
    \coordinate (c_coord) at (-0.3, 0.3); 
    \coordinate (d_coord) at (-1.1, -1); 
    \coordinate (e_coord) at (1.5, -1);  

    \draw[thick] (d_coord) -- (a2) -- (a1) -- (a4) -- (c_coord) -- (a3_coord) -- cycle;
    
    \draw[thick] (a2) -- (e_coord) -- (a4);

    \node[dot, fill=red] at (a1) {$a_1$};
    \node[dot, fill=blue] at (a2) {$a_2$};
    \node[dot, fill=red] at (a3_coord) {$c$};
    \node[dot, fill=green!80!black] at (a4) {$a_4$};
    \node[dot, fill=blue] at (c_coord) {$d_4$};
    \node[dot, fill=green!80!black] at (d_coord) {$d_2$};
    \node[dot, fill=red] at (e_coord) {$a_3$};

\end{tikzpicture}
            \end{center} 
            \item If $c$ is blue/green, without loss of generality, let it be blue. Hence, for $i = 1, 4$ there exists $d_i$ that is connected to $c$ and $a_i$ in order to block the edge $(c, a_i).$ Observe $(c,d_1,a_1,a_2,a_3,a_4,d_4)$ and $(c,d_1,a_1,a_4,d_4)$ cover the edges. 
            
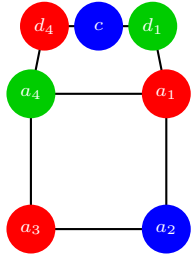
\begin{center}
\begin{tikzpicture}[
    scale=0.9,
    dot/.style={
        circle, 
        minimum size=6.5mm, 
        inner sep=0pt, 
        font=\tiny, 
        text=white,
        thick
    }
]

    \definecolor{green!80!black}{RGB}{0, 100, 0}

    \coordinate (a1) at (1, 1);
    \coordinate (a2) at (1, -1);
    \coordinate (a3) at (-1, -1);
    \coordinate (a4) at (-1, 1);
    
    \coordinate (c) at (0, 2);
    \coordinate (d1) at (.8, 2);
    \coordinate (d4) at (-.8, 2);

    \draw[thick] (a1) -- (a2) -- (a3) -- (a4) -- cycle;
    
    \draw[thick] (a4) -- (d4) -- (c) -- (d1) -- (a1);

    \node[dot, fill=red] at (a1) {$a_1$};
    \node[dot, fill=blue] at (a2) {$a_2$};
    \node[dot, fill=red] at (a3) {$a_3$};
    \node[dot, fill=green!80!black] at (a4) {$a_4$};
    
    \node[dot, fill=blue] at (c) {$c$};
    \node[dot, fill=green!80!black] at (d1) {$d_1$};
    \node[dot, fill=red] at (d4) {$d_4$};

\end{tikzpicture}
            \end{center}
        \end{enumerate}
    \end{enumerate}
\end{proof}
\begin{lemma} \label{2nd4}
    If the second coloring of the 4-cycle is a subgraph of a MPGC, then each edge is part of a 5-cycle. 
\end{lemma} 
\begin{proof}
    Note that the complement is disconnected. Let $c$ be a vertex such that it forms a path between $V_a$ and $V_b.$ $c$ must exist because of \cite{florez}. 

    We claim that $c$ is connected to at most 1 of the 4 from the original graph. $c$ can not be connected to both vertices in either $V_a$ or $V_b,$ or else it contradicts the fact that $c$ is connected to them in the original graph. $c$ also can not be connected to one from each because it will form a triangle. Hence, from here, we will assume that for every vertex $c$ is connected to, the vertex is not a part of the 4-cycle, unless stated otherwise.
    
    There are 2 cases, if $c$ is connected to 1 of the verticesz or if $c$ is connected to none. 
    \begin{enumerate}[leftmargin=0.5cm]
        \item Consider when there is one connection first. Without loss of generality, assume $c$ is connected to $a_1$. It follows that since $c$ can not be red, there must be another vertex $d$ blocking the edge $a_1, c.$ However, notice the 5-cycles.
        
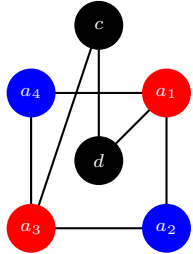
\begin{center}
\begin{tikzpicture}[
    scale=0.9,
    dot/.style={
        circle, 
        minimum size=6.5mm, 
        inner sep=0pt, 
        font=\tiny, 
        text=white,
        thick
    }
]

    \coordinate (a1) at (1, 1);
    \coordinate (a2) at (1, -1);
    \coordinate (a3) at (-1, -1);
    \coordinate (a4) at (-1, 1);
    
    \coordinate (c) at (0, 2);
    \coordinate (c1) at (0, 0); 

    \draw[thick] (a1) -- (a2) -- (a3) -- (a4) -- cycle;
    
    \draw[thick] (a3) -- (c) -- (c1) -- (a1);

    \node[dot, fill=red] at (a1) {$a_1$};
    \node[dot, fill=blue] at (a2) {$a_2$};
    \node[dot, fill=red] at (a3) {$a_3$};
    \node[dot, fill=blue] at (a4) {$a_4$};
    
    \node[dot, fill=black] at (c) {$c$};
    \node[dot, fill=black] at (c1) {$d$};

\end{tikzpicture}
        \end{center}
        \item Now, consider when there is no connection. We will do casework based on the color of $c.$ 
        \begin{enumerate}[leftmargin=0.5cm]
            \item Let $c$ be green. It follows that there exists a vertex $c_i$ that is connected to $a_i$ and $c$ so that it will block the edge $(a_i, c).$ Take $(a_1,d,c,a_3,a_2)$ and $(a_1,d,c,a_3,a_4)$ so each edge is covered. 
            
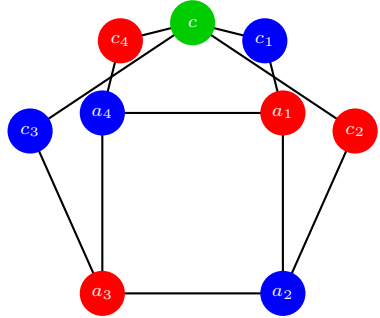
\begin{center}
\begin{tikzpicture}[
    scale=1.2,
    dot/.style={
        circle, 
        minimum size=6mm, 
        inner sep=0pt, 
        font=\tiny, 
        text=white,
        thick
    }
]

    \definecolor{green!80!black}{RGB}{0, 100, 0}

    \coordinate (a1) at (1, 1);
    \coordinate (a2) at (1, -1);
    \coordinate (a3) at (-1, -1);
    \coordinate (a4) at (-1, 1);
    
    \coordinate (c) at (0, 2);
    
    \coordinate (c1) at (0.8, 1.8);
    \coordinate (c2) at (1.8, 0.8);
    \coordinate (c3) at (-1.8, 0.8);
    \coordinate (c4) at (-0.8, 1.8);

    \draw[thick] (a1) -- (a2) -- (a3) -- (a4) -- cycle;
    
    \draw[thick] (c) -- (c1) -- (a1);
    \draw[thick] (c) -- (c2) -- (a2);
    \draw[thick] (c) -- (c3) -- (a3);
    \draw[thick] (c) -- (c4) -- (a4);

    \node[dot, fill=red] at (a1) {$a_1$};
    \node[dot, fill=blue] at (a2) {$a_2$};
    \node[dot, fill=red] at (a3) {$a_3$};
    \node[dot, fill=blue] at (a4) {$a_4$};
    
    \node[dot, fill=green!80!black] at (c) {$c$};
    
    \node[dot, fill=blue] at (c1) {$c_1$};
    \node[dot, fill=red] at (c2) {$c_2$};
    \node[dot, fill=blue] at (c3) {$c_3$};
    \node[dot, fill=red] at (c4) {$c_4$};

\end{tikzpicture}
            \end{center}
            \item Let $c$ be blue/red. Without loss of generality, let it be blue. It follows that there exists the vertices $c_1, c_3$ that blocks the edge $(c, a_1)$ and $(c, a_3).$ If $c_1=c_3,$ then $(a_3,a_2,a_1,c_1)$ and $(a_3,a_4,a_1,c_1)$ covers the edges. Hence, suppose they are distinct. 
            
            Look at the edge $(c_3, a_1).$ If it is blocked by a vertex, then there are 5-cycles(the vertex must be new since there would be triangles otherwise). 
            
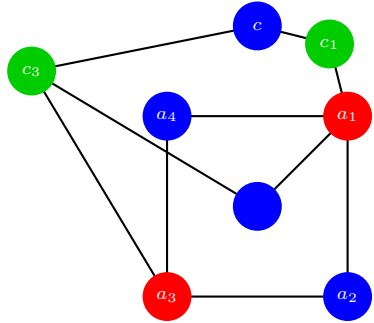
\begin{center}
\begin{tikzpicture}[
    scale=1.2,
    dot/.style={
        circle, 
        minimum size=6.5mm, 
        inner sep=0pt, 
        font=\tiny, 
        text=white,
        thick
    }
]

    \definecolor{green!80!black}{RGB}{0, 100, 0}

    \coordinate (a1) at (1, 1);
    \coordinate (a2) at (1, -1);
    \coordinate (a3) at (-1, -1);
    \coordinate (a4) at (-1, 1);
    
    \coordinate (c) at (0, 2);
    \coordinate (c1) at (0.8, 1.8);
    \coordinate (c3) at (-2.5, 1.5);
    \coordinate (d) at (0, 0);

    \draw[thick] (a1) -- (a2) -- (a3) -- (a4) -- cycle;
    
    \draw[thick] (c) -- (c1) -- (a1) -- (d) -- (c3);
    
    \draw[thick] (c) -- (c3) -- (a3);

    \node[dot, fill=red] at (a1) {$a_1$};
    \node[dot, fill=blue] at (a2) {$a_2$};
    \node[dot, fill=red] at (a3) {$a_3$};
    \node[dot, fill=blue] at (a4) {$a_4$};
    
    \node[dot, fill=blue] at (c) {$c$};
    \node[dot, fill=green!80!black] at (c1) {$c_1$};
    \node[dot, fill=green!80!black] at (c3) {$c_3$};
    \node[dot, fill=blue] at (d) {}; 

\end{tikzpicture}
            \end{center}
            Hence, the edge must exist.  
            
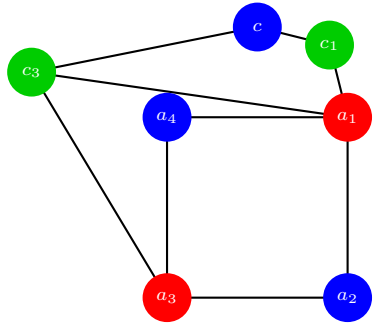
\begin{center}
\begin{tikzpicture}[
    scale=1.2,
    dot/.style={
        circle, 
        minimum size=6.5mm, 
        inner sep=0pt, 
        font=\tiny, 
        text=white,
        thick
    }
]

    \definecolor{green!80!black}{RGB}{0, 100, 0}

    \coordinate (a1) at (1, 1);
    \coordinate (a2) at (1, -1);
    \coordinate (a3) at (-1, -1);
    \coordinate (a4) at (-1, 1);
    
    \coordinate (c) at (0, 2);
    \coordinate (c1) at (0.8, 1.8);
    \coordinate (c3) at (-2.5, 1.5);

    \draw[thick] (a1) -- (a2) -- (a3) -- (a4) -- cycle;
    
    \draw[thick] (c) -- (c1) -- (a1) -- (c3);
    
    \draw[thick] (c) -- (c3) -- (a3);

    \node[dot, fill=red] at (a1) {$a_1$};
    \node[dot, fill=blue] at (a2) {$a_2$};
    \node[dot, fill=red] at (a3) {$a_3$};
    \node[dot, fill=blue] at (a4) {$a_4$};
    
    \node[dot, fill=blue] at (c) {$c$};
    \node[dot, fill=green!80!black] at (c1) {$c_1$};
    \node[dot, fill=green!80!black] at (c3) {$c_3$};

\end{tikzpicture}
            \end{center}
            However, observe the 4 cycles $(a_1, a_2, a_3, c_3), (a_1, a_4, a_3, c_3)$ of the first coloring type. It follows that by Lemma \ref{1st4}, we know that each edge of this graph must be part of a 5-cycle. 
        \end{enumerate}
    \end{enumerate}

\end{proof}
\begin{lemma} \label{4-cycle}
    In a MPGC, each edge of a 4-cycle is contained in a 5-cycle.
\end{lemma}
\begin{proof}
    Follows immediately from Lemma \ref{1st4} and \ref{2nd4}.
\end{proof}
\begin{lemma} \label{6-cycle}
    In a MPGC, each edge of a 6-cycle is contained in a 5-cycle.
\end{lemma}
\begin{proof}
    Note that by Lemma \ref{4-cycle}, we only need to prove that each edge is contained in either a 4 or 5-cycle, instead of just 5-cycle. We will do casework on the number of long diagonals in the 5-cycle. Let our MPGC graph be $\Gamma.$ 
    \begin{enumerate}[leftmargin=0.5cm]
        \item If there are no long diagonals, then adding the diagonal $(a_1, a_4)$ must result in either a contradiction in triangle-freeness or 3-colorability. 
        \begin{enumerate}[leftmargin=0.5cm]
            \item If it fails the triangle-free constraint, there exists $b \in V_\Gamma$ such that $(b, a_4), (b, a_1) \in E_\Gamma.$ 
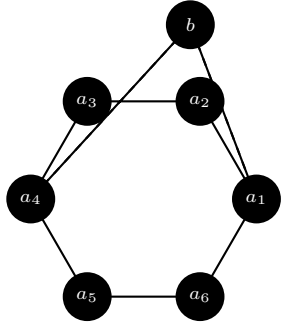
            \begin{center}
\begin{tikzpicture}[
    scale=1.5,
    dot/.style={
        circle, 
        minimum size=6.5mm, 
        inner sep=0pt, 
        font=\tiny, 
        text=white,
        thick,
        fill=black
    }
]

    \definecolor{green!80!black}{RGB}{0, 100, 0}

    \coordinate (a1) at (0:1);
    \coordinate (a2) at (60:1);
    \coordinate (a3) at (120:1);
    \coordinate (a4) at (180:1);
    \coordinate (a5) at (240:1);
    \coordinate (a6) at (300:1);
    
    \coordinate (b) at (75:1.6);

    \draw[thick, black] (a1) -- (a2) -- (a3) -- (a4) -- (b) -- (a1);
    
    \draw[thick, black] (a4) -- (a5) -- (a6) -- (a1);
    
    \draw[thick, black, dashed] (a1) -- (b) -- (a4);

    \node[dot] at (a1) {$a_1$};
    \node[dot] at (a2) {$a_2$};
    \node[dot] at (a3) {$a_3$};
    \node[dot] at (a4) {$a_4$};
    \node[dot] at (a5) {$a_5$};
    \node[dot] at (a6) {$a_6$};
    \node[dot] at (b) {$b$};

\end{tikzpicture}
            \end{center}
            However, observe the cycles, $(b, a_1, a_6, a_5, a_4), (b, a_1, a_2, a_3, a_4).$ 
            \item Now we handle the case where it fails the 3-colorability contraint. It follows that $(a_1,a_4)$, $(a_2,a_5)$, $(a_3,a_6)$ must be pairs of same colors. But then, this is isomorphic to $\Gamma(1, 1, 1, 1, 1, 1),$ and we are done by Lemma \ref{counterexample}.
            \begin{center}
\begin{tikzpicture}[
    scale=1.5,
    dot/.style={
        circle, 
        minimum size=6.5mm, 
        inner sep=0pt, 
        font=\tiny, 
        text=white,
        thick
    }
]

    \definecolor{green!80!black}{RGB}{0, 100, 0}

    \coordinate (a1) at (0:1);
    \coordinate (a2) at (60:1);
    \coordinate (a3) at (120:1);
    \coordinate (a4) at (180:1);
    \coordinate (a5) at (240:1);
    \coordinate (a6) at (300:1);

    \draw[thick] (a1) -- (a2) -- (a3) -- (a4) -- (a5) -- (a6) -- cycle;

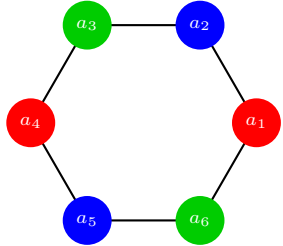
    \node[dot, fill=red] at (a1) {$a_1$};
    \node[dot, fill=blue] at (a2) {$a_2$};
    \node[dot, fill=green!80!black] at (a3) {$a_3$};
    \node[dot, fill=red] at (a4) {$a_4$};
    \node[dot, fill=blue] at (a5) {$a_5$};
    \node[dot, fill=green!80!black] at (a6) {$a_6$};

\end{tikzpicture}
            \end{center}
        \end{enumerate}
        \item If there is at least one diagonal. Without loss of generality, let one of them be the diagonal $(a_1, a_4).$ However, notice the 4-cycles $(a_1, a_2, a_3, a_4)$ and $(a_1, a_6 ,a_5, a_4)$.
    \end{enumerate}
\end{proof}
\begin{lemma}\label{7}
    In a MPGC, each edge of a 7-cycle is contained in a 5-cycle.
\end{lemma}
\begin{proof}
    By Lemma \ref{7-cycle}, we only need to prove it when at least one color shows up at most once. The only non-isomorphic coloring is $(r,g,b,g,b,g,b).$ 
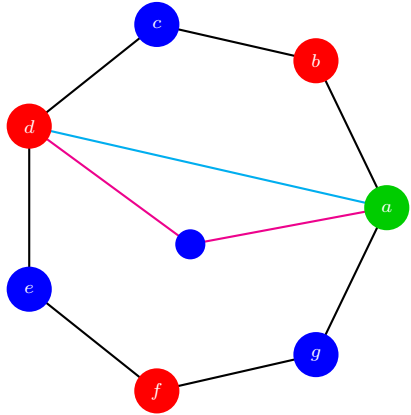
    \begin{center}
\begin{tikzpicture}[
    scale=2.5,
    dot/.style={
        circle, 
        minimum size=6mm, 
        inner sep=0pt, 
        font=\tiny, 
        text=white,
        thick
    }
]

    \definecolor{green!80!black}{RGB}{0, 100, 0}

    \foreach \i [count=\j from 0] in {a,b,c,d,e,f,g} {
        \coordinate (\i) at ({360/7 * \j}:1);
    }
    
    \coordinate (h) at ({360/7 * 5}:0.2);

    \draw[thick] (a) -- (b) -- (c) -- (d) -- (e) -- (f) -- (g) -- cycle;
    
    \draw[thick, cyan] (a) -- (d);
    
    \draw[thick, magenta] (a) -- (h) -- (d);

    \node[dot, fill=green!80!black] at (a) {$a$};
    \node[dot, fill=red]        at (b) {$b$};
    \node[dot, fill=blue]       at (c) {$c$};
    \node[dot, fill=red]        at (d) {$d$};
    \node[dot, fill=blue]       at (e) {$e$};
    \node[dot, fill=red]        at (f) {$f$};
    \node[dot, fill=blue]       at (g) {$g$};
    
    \node[dot, fill=blue, minimum size=4mm] at (h) {};

\end{tikzpicture}
    \end{center}
    Consider the edge $(a, d).$ If it exists, then by Lemma \ref{4-cycle}, we are done. If it does not exist, then there must exist a blue vertex blocking the edge. If the blue vertex was $e,$ we are done by same reasoning as earlier. If the vertex was $c$ or $g$, then there is a triangle. If it was a new vertex, we have Lemma \ref{6-cycle} which also finishes each edge. Thus, each edge of a 7-cycle is contained in a 5-cycle. 
\end{proof}
\begin{thm}\label{5-complement}
    In an MPGC, each edge is contained in a 5-cycle.
\end{thm}
\begin{proof}
    We will prove the edge $(u, v)$ is contained in a 5-cycle. By Theorem \ref{deg2}, there exists another vertex $x$ connected to $u.$ Similarly, there exists a vertex $y$ connected to $v.$ If $x = y,$ there is a triangle. Hence, assume that $x \neq y.$ By Theorem \ref{deg2}, there exists another vertex $z$ connected to $y.$ 
    
    If $x = z,$ we have the 4-cycle $(x, u, v, y),$ and by Lemma \ref{4-cycle}, we are done. 

    Assume $x \neq z.$ Now, because of Theorem \ref{dia23}, we know there is a path of length 1, 2, or 3 between $x$ and $z.$ However, in each case, there exists a 5, 6, or 7-cycle containing the edge $(u, v)$ and thus, by Lemma \ref{6-cycle}, \ref{7}, we are done(if any of the blocking vertices are pre-existing vertices, there are triangles). 
    
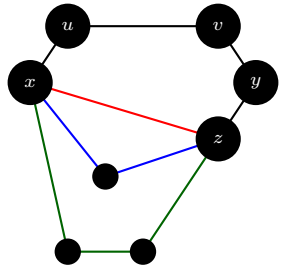
\begin{center}
\begin{tikzpicture}[
    scale=2,
    dot/.style={
        circle, 
        minimum size=6mm, 
        inner sep=0pt, 
        font=\tiny, 
        text=white,
        thick,
        fill=black
    }
]

    \definecolor{green!80!black}{RGB}{0, 100, 0}

    \coordinate (u) at (-0.5, 0);
    \coordinate (v) at (0.5, 0);
    \coordinate (x) at (-0.75, -0.375);
    \coordinate (y) at (0.75, -0.375);
    \coordinate (z) at (0.5, -0.75);
    
    \coordinate (a1) at (-0.25, -1);
    \coordinate (a2) at (-0.5, -1.5);
    \coordinate (a3) at (0, -1.5);

    \draw[thick] (x) -- (u) -- (v) -- (y) -- (z);
    
    \draw[thick, red] (x) -- (z);
    
    \draw[thick, blue] (x) -- (a1) -- (z);
    
    \draw[thick, green!80!black] (x) -- (a2) -- (a3) -- (z);

    \node[dot] at (x) {$x$};
    \node[dot] at (u) {$u$};
    \node[dot] at (v) {$v$};
    \node[dot] at (y) {$y$};
    \node[dot] at (z) {$z$};
    
    \node[dot, minimum size=3.5mm] at (a1) {};
    \node[dot, minimum size=3.5mm] at (a2) {};
    \node[dot, minimum size=3.5mm] at (a3) {};

\end{tikzpicture}
    \end{center}
\end{proof} 
\begin{remark}
    However, each edge of MPGs is not necessarily contained in a 5-cycle. Consider the diagram below, which is a 5-cycle with one duplicated vertex. It can be seen that the (blue)edge connecting the vertices is not contained in a 5-cycle.
    \begin{center}
\begin{tikzpicture}[
    scale=1.5,
    dot/.style={
        circle, 
        minimum size=4mm, 
        inner sep=0pt, 
        thick,
        fill=black
    }
]

    \foreach \i in {0,1,2,3,4} {
        \coordinate (a\i) at ({72 * \i}:1);
    }
    \coordinate (O) at (0,0);

    \draw[thick] (a1) -- (a2) -- (a3) -- (a4) -- (a0) -- cycle;
    
    \draw[thick] (a1) -- (O) -- (a3);
    
    \draw[thick, blue] (a2) -- (O);

    \node[dot] at (a0) {};
    \node[dot] at (a1) {};
    \node[dot] at (a2) {};
    \node[dot] at (a3) {};
    \node[dot] at (a4) {};
    \node[dot] at (O) {};

\end{tikzpicture}
    \end{center}
\end{remark}

\end{document}